\documentclass[a4paper,10pt]{amsart}

\usepackage{a4wide}
\usepackage[english]{babel}
\usepackage{cleveref}

\usepackage{amsfonts,amsfonts,mathtools}
\usepackage{graphicx}
\usepackage{epstopdf}
\usepackage{algorithmic}
\usepackage{colortbl}
\ifpdf
  \DeclareGraphicsExtensions{.eps,.pdf,.png,.jpg}
\else
  \DeclareGraphicsExtensions{.eps}
\fi

\usepackage{tikz,pgfplots,pgfplotstable} 
\pgfplotsset{compat=newest}
\usetikzlibrary{calc,3d,arrows,patterns,decorations}
\tikzset{>=stealth'}
\usetikzlibrary{positioning} 
\usepackage[abs]{overpic}
\usepackage[outline]{contour}
\contourlength{0.6pt}
\usetikzlibrary{decorations.text}
\tikzset{
  text deco/.style={postaction={decorate, decoration={text along path,#1}}}}
\usepackage{yfonts}

\usepackage{amsopn}
\DeclareMathOperator{\diag}{diag}

\newtheorem{theorem}{Theorem}

\newtheorem{lemma}[theorem]{Lemma}
\newtheorem{remark}[theorem]{Remark}
\newtheorem{corollary}[theorem]{Corollary}

\newcommand{\widebar}{\overline}
\newcommand{\ub}{\boldsymbol{u}}
\newcommand{\vb}{\boldsymbol{v}}
\newcommand{\pb}{\boldsymbol{p}}
\newcommand{\qb}{\boldsymbol{q}}
\newcommand{\fb}{\boldsymbol{f}}
\newcommand{\gb}{\boldsymbol{g}}
\newcommand{\yb}{\boldsymbol{y}}
\newcommand{\zb}{\boldsymbol{z}}
\newcommand{\wb}{\boldsymbol{w}}

\renewcommand{\eqref}{\cref}

\title[On the analysis of block smoothers]{On the analysis of block smoothers\\ for saddle point problems}
\author[L.~John, U.~R\"ude, B.~Wohlmuth, and W.~Zulehner]{Lorenz~John$^\MakeLowercase{a}$\and Ulrich~R\"ude$^\MakeLowercase{b}$\and Barbara~Wohlmuth$^\MakeLowercase{a}$\and Walter~Zulehner$^\MakeLowercase{c}$}
\date{}

\begin{document}

\maketitle
\vspace*{-1.5em}

\begin{center}
\begin{footnotesize}
$^a$ Institute for Numerical Mathematics, Technische Universit\"at M\"unchen, 85748 Garching b.~M\"unchen, Germany\\
$^b$ Institute of System Simulation, University Erlangen-Nuremberg, 91058 Erlangen, Germany\\
$^c$ Institute of Computational Mathematics, Johannes Kepler University Linz, 4040 Linz, Austria
\end{footnotesize}
\end{center}

\begin{abstract}
In this article, we discuss several classes of Uzawa smoothers for the application in multigrid methods in the context of saddle point problems. Beside commonly used variants, such as the inexact and block factorization version, we also introduce a new symmetric method, belonging to the class of Uzawa smoothers. For these variants we unify the analysis of the smoothing properties, which is an important part in the multigrid convergence theory. These methods are applied to the Stokes problem
for which all smoothers are implemented as pointwise relaxation methods.
Several numerical examples illustrate the theoretical results.
\end{abstract}

\bigskip
{\small

}
\section{Introduction}
Multigrid methods are considered as one of the most efficient solvers for the approximative solution of discretized (systems of) partial differential equations. These methods have been widely used in the past for different types of equations and applications, see, e.g., \cite{hackbusch85, trottenberg_2001} for a general overview and \cite{BWY1990, braess-sarazin_1997, BL_2011, brenner-li-sung_2014, gaspar-notay-osterlee-rodrigo_2014, OG_2008, schoeberl_2003, wittum_1990, WT_2009} for saddle point problems. Also with the increasing compute power these methods became more popular, see, e.g., \cite{gmeiner-huber-john-ruede-wohlmuth_2015,GRSWW2015, NN2015}.

In this paper, we discuss multigrid methods for problems whose discretization leads to linear systems in saddle point form
\begin{equation} \label{ABC}
  \mathcal{A} \begin{bmatrix} \ub \\ \pb \end{bmatrix} = \begin{bmatrix} \fb \\ \gb \end{bmatrix} 
  \quad \text{with} \quad
  \mathcal{A} = \begin{bmatrix}  A & B^\top \\ B & - C \end{bmatrix},
\end{equation}
where $A \in \mathbb{R}^{n\times n}$, $C \in \mathbb{R}^{m\times m}$ are symmetric and positive semi-definite, and $B \in \mathbb{R}^{m \times n}$ with $\ker A \cap \ker B = \{0\}$ and $\ker B^\top \cap \ker C = \{0\}$ and $n,m \in \mathbb{N}$. Note, the last conditions ensure that $\mathcal{A}$ is non-singular. In particular, we focus on multigrid methods for the Stokes problem in fluid mechanics. 

One essential component of efficient multigrid methods is the smoothing procedure, which consists of one step or a few steps of a non-expensive iterative method applied to \eqref{ABC}, whose role is to smooth the error. Here, we focus on stationary iterative methods of the form
\begin{equation} \label{PA}
  \begin{bmatrix} \ub_{k+1} \\ \pb_{k+1} \end{bmatrix}
   =
  \begin{bmatrix} \ub_{k} \\ \pb_{k} \end{bmatrix}
    + \mathcal{P}^{-1} \left(\begin{bmatrix} \fb \\ \gb \end{bmatrix} - \mathcal{A} \begin{bmatrix} \ub_{k} \\ \pb_{k} \end{bmatrix}\right),
\end{equation}
for the smoothing procedure with an appropriate choice for $\mathcal{P}$, such that the application of $\mathcal{P}^{-1}$ on a vector can be efficiently computed.
In particular this means that all formal matrix-blocks inversions
will be approximated by simple point-relaxation methods.

Convergence properties of iterative methods of the form \eqref{PA} have been studied in numerous articles, see, e.g., \cite{axelsson_2015, benzi05, notay14} and the references therein. Less known are the smoothing properties of such methods, which are essential for the convergence in multigrid theory  \cite{braess-sarazin_1997, brenner-li-sung_2014, olshanskii_2012, reusken_1991, schoeberl_2003, zulehner_2000}. For the analysis of the smoothing property for other iterative methods, see also \cite{brenner:1996,turek_1994, verfuerth:1984, verfuerth:1988}.

In this article, we aim to present a unified analysis for the smoothing property of several block preconditioners for saddle point problems \eqref{ABC}. One of the considered methods was already analyzed previously in \cite{schoeberl_2003}. Here, we prove the smoothing property of the so-called inexact Uzawa algorithm which has been successfully applied as a smoother in multigrid algorithms in the past, see, e.g., \cite{gaspar-notay-osterlee-rodrigo_2014, gmeiner-huber-john-ruede-wohlmuth_2015, larin-reusken_2008, maitre-musy-nigon_1984, schoeberl_2003}. 
Moreover, we present and analyze a new, symmetric smoother, which can be seen as a symmetrized variant of the inexact Uzawa method.

When applied to the Stokes system, all block-inversions will be
algorithmically approximated by point-relaxations
of Jacobi or Gauss-Seidel type.
Thus the algorithms are applicable in the case when a parallel
software structure for distributed memory systems
permits only restricted data access operations.
In this case, the distributed smoothers of
\cite{BL_2011} cannot be used though they may be more effective,
see, e.g., \cite{wang2013multigrid}.

This article is structured as follows: In \cref{sec:framework}, we introduce our notation and recall the key ingredients for the multigrid convergence analysis. \Cref{sec:classes} gives an overview on the different Uzawa-type smoothers and their relations to each other. In \cref{sec:smoothing_property}, we present the analysis of the smoothing property for the individual methods. Finally, in \cref{sec:numerics}, we present numerical tests, applied to the Stokes problem, illustrating the obtained theoretical results.

\section{Notation and multigrid analysis}\label{sec:framework}

In this section, we introduce some notations and recall the key ingredients for the multigrid convergence analysis.

Let $V$ and $Q$ be real-valued Hilbert spaces. We assume that the linear system \eqref{ABC} results from the discretization of a saddle point problem of the following form: Find $(u,p) \in V \times Q$ such that
\begin{equation}
\begin{alignedat}{3} \label{contabc}
  a(u,v) & + b(v,p) && = f(v) & \quad & \forall \, v \in V,\\
  b(u,q) & - c(p,q) && = g(q) & \quad & \forall \, q \in Q,
\end{alignedat}
\end{equation}
with the bilinear forms $a \colon V \times V \rightarrow \mathbb{R}$, $b \colon V \times Q \rightarrow \mathbb{R}$, $c \colon Q \times Q \rightarrow \mathbb{R}$, and linear functionals $f \in V^\ast$, $g \in Q^\ast$. Notice, the symbol $\ast$ indicates the dual Hilbert space. Equivalently, the mixed variational problem \eqref{contabc} can be written as a variational problem in the product space: Find $(u,p) \in V \times Q$ such that
\begin{equation} \label{contB}
  \mathcal{B}((u,p),(v,q)) = \mathcal{F}(v,q) \quad \forall \, (v,q) \in V \times Q,
\end{equation}
with the bilinear form $\mathcal{B}((u,p),(v,q)) = a(u,v) + b(v,p) + b(u,q) - c(p,q)$ and the linear functional $\mathcal{F}(v,q) = f(v) + g(q)$.

For the discretization of problem \eqref{contB} a sequence of appropriate finite-dimensional subspaces $V_\ell \subset V$ and $Q_\ell \subset Q$, $\ell = 0,1,\ldots,L$, (typically finite element spaces) have to be chosen, corresponding to a sequence of increasingly finer meshes of the computational domain. We assume for simplicity that these subspaces are nested, i.e.,
\[
  V_{\ell-1} \subset V_\ell 
  \quad \text{and} \quad
  Q_{\ell-1} \subset Q_\ell \quad \forall\, \ell=1, \dots, L.
\]
The approximate solution $(u_\ell, p_\ell) \in V_\ell\times Q_\ell$ on level $\ell$ is given by the discrete variational problem
\begin{equation} \label{discrB}
  \mathcal{B}_\ell((u_\ell,p_\ell),(v_\ell,q_\ell)) = \mathcal{F}_\ell(v_\ell,q_\ell) 
  \quad \forall \, (v_\ell,q_\ell) \in V_\ell \times Q_\ell,
\end{equation}
with $\mathcal{B}_\ell((u_\ell,p_\ell),(v_\ell,q_\ell)) = a(u_\ell,v_\ell) + b(v_\ell,p_\ell) + b(u_\ell,q_\ell) - c_\ell(p_\ell,q_\ell),$ and $\mathcal{F}_\ell(v_\ell,q_\ell) = f(v_\ell) + g_\ell(q_\ell)$, where $c_\ell(\cdot,\cdot)$ is a bilinear form on the discrete space, and $g_\ell \in Q_\ell^*$.

\subsection{Multigrid methods and their convergence analysis}
In the following we introduce the notation for the multigrid method and recall some crucial results. Let $(u_\ell^{(0)},p_\ell^{(0)}) \in V_\ell \times Q_\ell$ be a given approximation of the solution of \eqref{discrB} on level $\ell \ge 1$. Then one step of a multigrid algorithm at level $\ell$ is recursively defined by the following three steps:
\begin{enumerate}
\item \textit{Pre-smoothing}: compute $(u_\ell^{(\nu_1)},p_\ell^{(\nu_1)})$ by an iterative method of the form
\[
  (u_\ell^{(k+1)},p_\ell^{(k+1)}) = {\mathcal S}_\ell \, (u_\ell^{(k)},p_\ell^{(k)}) ,
  \quad \forall\, k = 0,\ldots,\nu_1-1.
\]
\item \textit{Coarse grid correction}: set
\[
  \mathcal{R}_\ell(v_\ell,q_\ell)
     =  \mathcal{F}_\ell(v_\ell,q_\ell) -
        {\mathcal B}_\ell \big( (u_\ell^{(\nu_1)},p_\ell^{(\nu_1)}),(v_\ell,q_\ell)
	               \big),\quad \forall\, (v_\ell,q_\ell) \in V_\ell \times Q_\ell,
\]
and let $(\bar{w}_{\ell-1},\bar{r}_{\ell-1}) \in V_{\ell-1} \times Q_{\ell-1}$ satisfy
\begin{equation} \label{cgc}
  {\mathcal B}_{\ell-1}((\bar{w}_{\ell-1},\bar{r}_{\ell-1}),(v_{\ell-1},q_{\ell-1})) = \mathcal{R}_\ell(v_{\ell-1},q_{\ell-1}),
\end{equation}
for all $(v_{\ell-1},q_{\ell-1}) \in V_{\ell-1} \times Q_{\ell-1}$.\\
If $\ell = 1$, solve equation \eqref{cgc} exactly and set $(w_{\ell-1},r_{\ell-1}) = (\bar{w}_{\ell-1},\bar{r}_{\ell-1})$.\\
If $\ell > 1$, compute an approximation $(w_{\ell-1},r_{\ell-1})$ to $(\bar{w}_{\ell-1},\bar{r}_{\ell-1})$ by applying $\gamma$ steps 
of the multigrid algorithm at level $\ell-1$ applied to \eqref{cgc} with initial guess $(w_{\ell-1}^{(0)},r_{\ell-1}^{(0)}) = 0$.\\
Set 
\begin{equation} \label{multigrid}
   (u_\ell^{(\nu_1+1)},p_\ell^{(\nu_1+1)}) 
    = (u_\ell^{(\nu_1)},p_\ell^{(\nu_1)}) + (w_{\ell-1},r_{\ell-1}).
\end{equation}
\item \textit{Post-smoothing}: 
compute $(u_\ell^{(\nu +1)},p_\ell^{(\nu+1)})$ by an iterative method of the form
\[
  (u_\ell^{(k+1)},p_\ell^{(k+1)}) = {\mathcal S}_\ell^* \, (u_\ell^{(k)},p_\ell^{(k)}) ,
  \quad \forall\, k = \nu_1+1,\ldots,\nu_1+\nu_2 = \nu.
\]
Set $(u_\ell^{(\nu +1)},p_\ell^{(\nu+1)})$ as the next iterate after one step of the multigrid method.
\end{enumerate}

\noindent Recall, for $\gamma=1$ we obtain the $V$-cycle and for $\gamma=2$ the $W$-cycle.  

For the multigrid convergence analysis of saddle point problems we follow the standard approach. We refer to \cite{hackbusch85} as a general reference and to \cite{brenner:1996,verfuerth:1984} for the approach adapted to saddle point problems. In the particular case of the Stokes problem we refer e.g. to \cite{brenner-li-sung_2014, schoeberl_2003, verfuerth:1988}.
This approach is based on the introduction of an appropriate mesh-dependent norm on the product space, denoted by 
\[ 
  \|(v_\ell,q_\ell)\|_{0,\ell}
  \quad \forall \, (v_\ell,q_\ell) \in V_\ell \times Q_\ell.
\]
The choice of this norm depends on the particular saddle point problem and its approximation spaces, this will be discussed in \cref{sec:numerics}. Associated with this norm and the bilinear form of the discrete problem a second mesh-dependent norm is introduced by
\begin{equation} \label{discr2ell}
  \|(w_\ell,r_\ell)\|_{2,\ell} = \sup_{0 \neq (v_\ell,q_\ell) \in V_\ell \times Q_\ell} \frac{\mathcal{B}_\ell((w_\ell,r_\ell),(v_\ell,q_\ell))}{\|(v_\ell,q_\ell)\|_{0,\ell}}
  \quad \forall \, (w_\ell,r_\ell) \in V_\ell \times Q_\ell.
\end{equation}

For simplicity, we restrict in this article the convergence analysis of multigrid methods to the case of pre-smoothing only, i.e. $\nu_2 = 0$. The crucial properties for the convergence analysis are the \textit{smoothing property}
\begin{equation} \label{smoothProp}
  \|(u_\ell^{(\nu)},p_\ell^{(\nu)}) - (u_\ell,p_\ell)\|_{2,\ell} \le c_S \, \eta(\nu) \, \|(u_\ell^{(0)},p_\ell^{(0)}) - (u_\ell,p_\ell)\|_{0,\ell}
\end{equation}
and the \textit{approximation property}
\begin{equation} \label{approxProp}
   \|(\bar{u}_\ell^{(\nu+1)},\bar{p}_\ell^{(\nu+1)}) - (u_\ell,p_\ell)\|_{0,\ell} \le c_A \, \|(u_\ell^{(\nu)},p_\ell^{(\nu)}) - (u_\ell,p_\ell)\|_{2,\ell}
\end{equation}
with
\begin{equation} \label{twogrid} 
  (\bar{u}_\ell^{(\nu+1)},\bar{p}_\ell^{(\nu+1)}) = (u_\ell^{(\nu)},p_\ell^{(\nu)}) + (\bar{w}_{\ell-1},\bar{r}_{\ell-1}),
\end{equation}
where $(\bar{w}_{\ell-1},\bar{r}_{\ell-1})$ is the exact solution of the coarse grid correction equation \eqref{cgc}. The constants $c_S$ and $c_A$, and the function $\eta(\nu)$ are supposed to be independent of the level $\ell$ and
\[
  \eta(\nu) \to 0 \quad \text{as} \quad  \nu \to \infty.
\] 
Recall, if \eqref{multigrid} is replaced by \eqref{twogrid} in the original multigrid method from above, one obtains the (non-recursive) two-grid method, which involves only two levels of the discretization, level $\ell$ and the next coarser level $\ell-1$. It is immediately clear that smoothing property and approximation property imply the convergence of the two-grid method with contraction rate $q = c_A\, c_S\, \eta(\nu) < 1$ uniformly in $\ell$ provided the number of smoothing steps $\nu$ is sufficiently large. 

The uniform convergence of the original multigrid method for $\gamma=2$ ($W$-cycle) follows then by standard arguments, see \cite{hackbusch85}, but as it is typical for this approach, it requires
regularity assumptions.

\begin{remark}
For the extension of the convergence analysis to the general case which includes post-smoothing, i.e. $\nu_2 \neq 0$, we refer to \cite[Subsection 6.1.5]{hackbusch85}.
\end{remark}

\subsection{Matrix-vector notation}
Let $(v_\ell, q_\ell) \in V_\ell \times Q_\ell$, and denote the dimensions by $n_\ell = \dim V_\ell$, $m_\ell = \dim Q_\ell$. Then, the isomorphism $(v_\ell, q_\ell) \leftrightarrow (\vb_\ell, \qb_\ell) \in \mathbb{R}^{n_\ell} \times \mathbb{R}^{m_\ell}$ is valid, where the latter denote the vectors of coefficients relative to given bases in $V_\ell$ and $Q_\ell$ (typically the nodal bases in case of finite element spaces). 

Consequently, the discrete problem \eqref{discrB} in matrix-vector notation reads as \eqref{ABC}, with matrices $A_\ell \in \mathbb{R}^{n_\ell\times n_\ell}$, $B_\ell \in \mathbb{R}^{m_\ell \times n_\ell}$, $C_\ell \in \mathbb{R}^{m_\ell\times m_\ell}$, and vectors $\fb_\ell \in \mathbb{R}^{n_\ell}$, $\gb_\ell \in \mathbb{R}^{m_\ell}$, additionally indicated by the refinement level $\ell$ and defined by the discrete (bi-)linear forms.

For the mesh-dependent norm $\|(v_\ell, q_\ell)\|_{0,\ell}$, used in the convergence analysis, we assume that it is represented by a symmetric and positive definite matrix $\mathcal{L}_\ell \in \mathbb{R}^{(n_\ell+m_\ell)\times (n_\ell+m_\ell)}$ as
\begin{equation} \label{discr0}
  \| y_\ell \|_{0,\ell} = \langle \mathcal{L}_\ell \yb_\ell, \yb_\ell \rangle^{1/2} \equiv \| \yb_\ell \|_{\mathcal{L}_\ell}
  \quad \text{with} \ y_\ell = (v_\ell,q_\ell), \ \yb_\ell = \begin{bmatrix} \vb_\ell \\ \qb_\ell \end{bmatrix},
\end{equation}
where $\langle .,. \rangle$ denotes the Euclidean inner product. For the associated Euclidean norm of vectors as well as for the spectral norm of matrices, the symbol $\| \cdot \|$ is used. Hence, the second mesh-dependent norm has the following representation
\begin{equation} \label{discr2}
  \| z_\ell \|_{2,\ell} 
     = \sup_{0 \neq \yb_\ell \in \mathbb{R}^{n_\ell+m_\ell}} \frac{\langle\mathcal{A}_\ell \zb_\ell, \yb_\ell \rangle}{\|\yb_\ell\|_{\mathcal{L}_\ell}}
  \quad \text{with} \ z_\ell = (w_\ell,r_\ell), \ \zb_\ell = \begin{bmatrix} \boldsymbol{w}_\ell \\ \boldsymbol{r}_\ell \end{bmatrix}.
\end{equation}

Furthermore, we use throughout the article the following notation. For symmetric matrices $M$ and $N$: $N > M$ iff $N-M$ is positive definite, and $N \geq M$ iff $N-M$ is positive semi-definite.

\section{Several classes of block smoothers}\label{sec:classes}

In the following, we introduce several classes of iterative methods as candidates for the smoothing procedure and discuss their relations. Here and in the next section, we drop the index $\ell$ identifying the level of discretization for sake of simplicity.

For the smoothing procedure, we consider iterative methods of the form \eqref{PA}, where $\mathcal{P}$ is addressed as preconditioner. In particular, we discuss preconditioners with the same block structure as the system matrix $\mathcal{A}$  and follow a similar classification as given in \cite{notay14}.

The first three classes are indefinite block {\em diagonal}, {\em lower} and {\em upper} block triangular preconditioners
\[
  \mathcal{P}_d 
   = \begin{bmatrix}
       \hat{A} & 0 \\ 0 & -\hat{S} 
     \end{bmatrix}, \quad
  \mathcal{P}_\ell 
   = \begin{bmatrix}
       \hat{A} & 0 \\ B & -\hat{S} 
     \end{bmatrix}
  \quad \text{and} \quad
  \mathcal{P}_u 
   = \mathcal{P}_\ell^\top 
   = \begin{bmatrix}
       \hat{A}^\top & B^\top \\ 0 & -\hat{S}^\top 
     \end{bmatrix}.
\]
Here, $\hat{A}$ and $\hat{S}$ denote non-singular approximations of $A$ and the Schur complement $S = C + B A^{-1} B^\top$, respectively. In the context of a smoothing procedure for multigrid methods we typically think of approximations $\hat{A}$ and $\hat{S}$ representing simple iterative methods like the Richardson method, the Jacobi or the forward, backward or symmetric Gauss--Seidel methods, although there is no special restriction of this kind in the forthcoming smoothing analysis.

The iterative method \eqref{PA} with $\mathcal{P} = \mathcal{P}_\ell$ is called an inexact Uzawa method. 
The choice $\mathcal{P} = \mathcal{P}_u$ in \eqref{PA} corresponds to the associated adjoint iterative method, see, e.g., \cite{hackbusch94} for the general concept of adjoint iterative methods. 

Motivated by the exact block factorization of $\mathcal{A}$ for non-singular $A$ given by
\begin{equation} \label{LDU}
  \mathcal{A}
   = \begin{bmatrix}
       I & 0 \\ B A^{-1} & I 
     \end{bmatrix}
     \begin{bmatrix}
       A & 0 \\ 0 & -S 
     \end{bmatrix}     
     \begin{bmatrix}
       I & A^{-1} B^\top \\ 0 & I 
     \end{bmatrix},
\end{equation}
with the same approximations $\hat{A}$ for $A$ in the block triangular and the block diagonal terms, leads to the fourth class of block approximate {\em factorization}  preconditioners of the form
\[
  \mathcal{P}_f 
   = \begin{bmatrix}
       I & 0 \\ B \hat{A}^{-1} & I 
     \end{bmatrix}
     \begin{bmatrix}
       \hat{A} & 0 \\ 0 & -\hat{S} 
     \end{bmatrix}     
     \begin{bmatrix}
       I & \hat{A}^{-1} B^\top \\ 0 & I 
     \end{bmatrix}.
\]
Using instead different approximations for the block triangular and the block diagonal terms in \eqref{LDU}, motivates the class of {\em symmetric} block preconditioners 
\[
  \mathcal{P}_s 
   = \begin{bmatrix}
       I & 0 \\ B \hat{A}^{-1} & I 
     \end{bmatrix}
     \begin{bmatrix}
       \hat{A}_s & 0 \\ 0 & - \hat{S} 
     \end{bmatrix}     
     \begin{bmatrix}
       I & \hat{A}^{-\top} B^\top \\ 0 & I 
     \end{bmatrix} 
    \quad\text{with}\quad
   \hat{A}_s = \hat{A} (\hat{A} + \hat{A}^\top - A)^{-1} \hat{A}^\top.
\]

Notice, the last class can be also motivated by the symmetric preconditioner $\mathcal{P}_\text{sym}$, resulting from the product iteration of the inexact Uzawa method and its adjoint iterative method, which is given by the condition
\[
  I - \mathcal{P}_\text{sym}^{-1} \mathcal{A} 
   = (I - \mathcal{P}_u^{-1} \mathcal{A})(I - \mathcal{P}_\ell^{-1} \mathcal{A}) ,
\]
and is equivalent to
\[
  \mathcal{P}_\text{sym} = \mathcal{P}_\ell (\mathcal{P}_\ell + \mathcal{P}_u - \mathcal{A})^{-1} \mathcal{P}_u ,
\]
see, e.g., \cite{hackbusch94} for the general concept of symmetric iterative methods. The symmetric preconditioner can then be expressed by
\begin{equation*}
  \mathcal{P}_\text{sym}
   = \begin{bmatrix}
       I & 0 \\ B \hat{A}^{-1} & I 
     \end{bmatrix}
     \begin{bmatrix}
       \hat{A}_s & 0 \\ 0 & -\hat{S}_s 
     \end{bmatrix}     
     \begin{bmatrix}
       I & \hat{A}^{-\top} B^\top \\ 0 & I 
     \end{bmatrix},
\end{equation*}
with $\hat{A}_s = \hat{A} (\hat{A} + \hat{A}^\top - A)^{-1} \hat{A}^\top$ and $\hat{S}_s = \hat{S} (\hat{S} + \hat{S}^\top - C)^{-1} \hat{S}^\top$. 
Observe that $\mathcal{P}_\text{sym}$ does not explicitly depend on the matrix $\hat{S}$ but only on $\hat{S}_s$. Replacing $\hat{S}_s$ by a general approximation (with a slight abuse of notation again denoted by) $\hat{S}$ leads to $\mathcal{P}_s$, which in this sense is more general than $\mathcal{P}_\text{sym}$. 

Observe,  all five classes $\mathcal{P}_d$, $\mathcal{P}_\ell$, $\mathcal{P}_u$, $\mathcal{P}_f$, and $\mathcal{P}_s$ for $\mathcal{A}$ depend on the choice of the two matrices $\hat{A}$ and $\hat{S}$. When necessary, we will explicitly express this by using $\hat{A}$ and $\hat{S}$ as arguments. In particular, we will use the notations $\mathcal{P}_\ell(\hat{A},\hat{S})$ and  $\mathcal{P}_s(\hat{A},\hat{S})$. Throughout the paper it is assumed for all five classes that:
\begin{itemize}
\item $\hat{A}$ is non-singular and $\hat{S}$ is symmetric and positive definite,
\item in the case of $\mathcal{P}_s$ we assume additionally that
\[
  \hat{A} + \hat{A}^\top > A,
\]
which ensures that $\hat{A}_s$ is symmetric and positive definite. 
\item in case of $\mathcal{P}_d$, $\mathcal{P}_\ell$, $\mathcal{P}_u$, and $\mathcal{P}_f$ we additionally assume that $\hat{A}$ is symmetric and positive definite.
\end{itemize}

\subsection{Relation of the smoothers}
Let us briefly discuss the relation of the iterative method \eqref{PA} for the different choices of $\mathcal{P}$, as previously introduced. For the computational efficiency of a multigrid method, low computational costs for the smoothing procedure are an essential factor. Besides the evaluation of the residual which, in each step, roughly costs one matrix-vector multiplication for each of the matrices $A$, $B$, $B^\top$ and $C$, the additional costs are different for the different classes. One step of \eqref{PA} with $\mathcal{P} = \mathcal{P}_d$ reads
\begin{align*}
   \ub_{k+1} & = \ub_k + \hat{A}^{-1}(\fb - A \ub_k - B^\top \pb_k), \\
   \pb_{k+1} & = \pb_k - \hat{S}^{-1}(\gb - B \ub_k + C \pb_k ),
\end{align*}
and one step of the inexact Uzawa method $\mathcal{P} = \mathcal{P}_\ell$ reads
\begin{align*}
   \ub_{k+1} & = \ub_k + \hat{A}^{-1}(\fb - A \ub_k - B^\top \pb_k), \\
   \pb_{k+1} & = \pb_k - \hat{S}^{-1}(\gb - B \ub_{k+1} + C \pb_k).
\end{align*}
Consequently, the computational costs of $\mathcal{P}_\ell$ are roughly identical with those of the block diagonal version $\mathcal{P}_d$. The same is true for the adjoint inexact Uzawa method $\mathcal{P} = \mathcal{P}_u$. Let us consider now $\mathcal{P} = \mathcal{P}_f$, where one step of the method reads
\begin{align*}
   \ub_\ast    & = \ub_k + \hat{A}^{-1}(\fb - A \ub_k - B^\top \pb_k), \\
   \pb_{k+1} & = \pb_k - \hat{S}^{-1}(\gb - B \ub_\ast + C \pb_k), \\
   \ub_{k+1} & = \ub_k + \hat{A}^{-1}(\fb - A \ub_k - B^\top \pb_{k+1}).
\end{align*}
Note that the computation can be arranged in such a way that, compared to the block diagonal and block triangular version, it costs one extra matrix-vector multiplication with $\hat{A}^{-1}$. Finally, one step for $\mathcal{P} = \mathcal{P}_s$ reads
\begin{align*}
   \ub_\ast    & = \ub_k + \hat{A}^{-1}(\fb - A \ub_k - B^\top \pb_k), \\
   \pb_{k+1} & = \pb_k - \hat{S}^{-1}(\gb - B \ub_\ast + C \pb_k), \\
   \ub_{k+1} & = \ub_\ast + \hat{A}^{-\top}(\fb - A \ub_\ast - B^\top \pb_{k+1}).
\end{align*}
In contrast to the case $\mathcal{P} = \mathcal{P}_f$ the intermediate value $\ub_\ast$ enters the third line.

At first sight it might look that $\mathcal{P}_s$ is the most costly scheme. Let us discuss this in more detail in the following remark.

\begin{remark}\label{relation_Pl_Ps}
We consider $\hat{A}^\top$ and $\hat{A}_s$, e.g. the associated matrices of the backward- and the symmetric Gauss--Seidel method. Several compositions of $\mathcal{P}_\ell(\hat{A}_s, \hat{S})$ and $\mathcal{P}_s(\hat{A}^\top, \hat{S})$ differ only in the first and last step of the application by $\hat{A}$, see \cref{F:smoother} for illustration. This suggests, that the methods are closely related to each other. We will employ this fact in the analysis of the smoothing property.

Moreover, in the implementation of $\mathcal{P}_s(\hat{A}^\top, \hat{S})$ the calculation of $\fb - B^\top \pb_{k+1}$ in the residual for the next iteration can be avoided and thus the computational cost can be reduced. Note, this idea can be also applied in the case of $\mathcal{P} = \mathcal{P}_f$.
\end{remark}

\renewcommand{\arraystretch}{0.3}
\begin{figure}[h!]	
\centering
\def\cc{\cellcolor[gray]{0.95}}
\def\len{1.5em}
\begin{tabular}{p{\len}p{\len}p{\len}p{\len}p{\len}p{\len}p{\len}p{\len}p{\len}p{\len}p{\len}p{\len}p{\len}p{\len}}
\cline{2-7}\cline{9-11}\\[-0.75ex]
&\multicolumn{1}{|c}{\cc} & \cc & \multicolumn{1}{c|}{\cc} & \cc & \cc & \multicolumn{1}{c|}{\cc} & &\multicolumn{1}{|c}{\cc} & \cc & \multicolumn{1}{c|}{\cc}\\[-0.5ex]
&\multicolumn{1}{|c}{\cc$\hat{A}$} & \cc$\hat{A}^\top$ & \multicolumn{1}{c|}{\cc$\hat{S}$} & \cc$\hat{A}$ & \cc$\hat{A}^\top$ & \multicolumn{1}{c|}{\cc$\hat{S}$} & $\cdots$ & \multicolumn{1}{|c}{\cc$\hat{A}$} & \cc$\hat{A}^\top$ & \multicolumn{1}{c|}{\cc$\hat{S}$}\\[0.125ex] \cline{2-7}\cline{9-11}\\ [-1.2ex]
\multicolumn{5}{c}{$\underbrace{\hspace*{8em}}_{\displaystyle\mathcal{P}_\ell(\hat{A}_s, \hat{S})=\hspace*{5em}}$} \\[-2ex]
& \multicolumn{5}{c}{$\overbrace{\hspace*{8em}}^{\displaystyle\hspace*{5em}=\mathcal{P}_s(\hat{A}^\top, \hat{S})}$} \\[-0.1ex]
\cline{3-8}\cline{10-12}\\[-0.75ex] 
& & \multicolumn{1}{|c}{\cc} & \cc & \multicolumn{1}{c|}{\cc} & \cc & \cc & \multicolumn{1}{c|}{\cc} & &\multicolumn{1}{|c}{\cc} & \cc & \multicolumn{1}{c|}{\cc}\\[-0.5ex]
& & \multicolumn{1}{|c}{\cc$\hat{A}^\top$} & \cc$\hat{S}$ & \multicolumn{1}{c|}{\cc$\hat{A}$} & \cc$\hat{A}^\top$ & \cc$\hat{S}$ & \multicolumn{1}{c|}{\cc$\hat{A}$} & $\cdots$ & \multicolumn{1}{|c}{\cc$\hat{A}^\top$} & \cc$\hat{S}$ & \multicolumn{1}{c|}{\cc$\hat{A}$} & \\[0.125ex] \cline{3-8}\cline{10-12}\\ [1ex] 
\end{tabular}
\caption{Sequence of matrices whose inverse are applied for $\nu$ steps, illustrating the relation of $\mathcal{P}_\ell(\hat{A}_s, \hat{S})$ and $\mathcal{P}_s(\hat{A}^\top, \hat{S})$.}
\label{F:smoother}
\end{figure}

\renewcommand{\arraystretch}{1.0}
\begin{remark}
For solving saddle point problems also the Braess--Sarazin smoother \textnormal{\cite{braess-sarazin_1997}} is applicable, which is defined by
\begin{align*}
\mathcal{P}_{BS} = 
\begin{bmatrix}
D_\alpha & B^\top \\
B & -C
\end{bmatrix},
\end{align*}
where $D_\alpha = \alpha\, \diag A$ with $\alpha \in \mathbb{R}$. One step of the iterative scheme \eqref{PA} with $\mathcal{P} = \mathcal{P}_{BS}$ then reads
\begin{align*}
   \ub_\ast    & = \ub_k + D_\alpha^{-1}(\fb - A \ub_k - B^\top \pb_k), \\
   \pb_{k+1} & = \pb_k - (C + B D_\alpha^{-1} B^\top)^{-1}(\gb - B \ub_\ast + C \pb_k), \\
   \ub_{k+1} & = \ub_k + D_\alpha^{-1}(\fb - A \ub_k - B^\top \pb_{k+1}).
\end{align*}
Thus this method is a special variant of the block approximate factorization scheme with $\mathcal{P} = \mathcal{P}_f$, where $\hat{A} = D_\alpha$ and $\hat{S} = C + B D_\alpha^{-1} B^\top$.
\end{remark}

\section{The smoothing property}\label{sec:smoothing_property}

In this section, we present a unified analysis of the smoothing property \eqref{smoothProp} for the individual schemes. We do so by using a mesh-dependent norm, which is represented by a symmetric and positive definite matrix $\mathcal{L}$. Recall the index $\ell$, indicating the level of discretization is neglected.

Firstly, we introduce the iteration matrix (or error propagation matrix) associated with the iterative method \eqref{PA} by
\[
  \mathcal{M}
     = I - \mathcal{P}^{-1} \mathcal{A}.
\]
In particular, for the five classes, the corresponding iteration matrices are denoted by $\mathcal{M}_d$, $\mathcal{M}_\ell$, $\mathcal{M}_u$, $\mathcal{M}_f$, and $\mathcal{M}_s$. 

For some matrix $\mathcal{K} \in \mathbb{R}^{(n+m) \times (n+m)}$, we denote the norm of the bilinear $\langle \mathcal{K} \zb,\yb \rangle$ by
\[
  \|\mathcal{K}\|_{\mathcal{L} \times \mathcal{L}} 
    = \sup_{0 \neq \zb,\yb \in \mathbb{R}^{n+m}} \frac{\langle\mathcal{K} \zb,\yb\rangle}{\|\zb\|_\mathcal{L} \|\yb\|_\mathcal{L}}.
\]
Using \eqref{discr0} and \eqref{discr2} the smoothing property \eqref{smoothProp} can be rewritten in matrix-vector notation as
\[
  \|\mathcal{A} \mathcal{M}^\nu\|_{\mathcal{L} \times \mathcal{L}} \le c_S \, \eta(\nu).
\]

The following lemmata provide helpful  estimates for the forthcoming analysis of the smoothing property.

\begin{lemma}\label{first}
Let  $\widebar{\mathcal{M}} = \mathcal{X} \mathcal{M} \mathcal{X}^{-1}$
for some non-singular matrix $\mathcal{X}$. Then there hold the estimates
\begin{equation} \label{firstestimate}
  \|\mathcal{A} \mathcal{M}^\nu\|_{\mathcal{L} \times \mathcal{L}}
    \le \|(I - \widebar{\mathcal{M}}) \widebar{\mathcal{M}}^\nu \|
        \, \|\mathcal{X}^{-\top} \mathcal{P} \mathcal{X}^{-1} \| 
        \, \|\mathcal{X}^\top \mathcal{X} \|_{\mathcal{L} \times \mathcal{L}},
\end{equation}
and
\begin{equation} \label{secondestimate}
  \|\mathcal{A} \mathcal{M}^\nu\|_{\mathcal{L} \times \mathcal{L}}
    \le \|(I - \widebar{\mathcal{M}}) \widebar{\mathcal{M}}^{\nu-1} \|
        \, \|\mathcal{X}^{-\top} (\mathcal{P} - \mathcal{A}) \mathcal{X}^{-1} \| 
        \, \|\mathcal{X}^\top \mathcal{X} \|_{\mathcal{L} \times \mathcal{L}}  .
\end{equation}
\end{lemma}

\begin{proof}
Since $\mathcal{A} = \mathcal{P}(I - \mathcal{M})$, it follows that
\[
  \mathcal{A} \mathcal{M}^\nu
    = \mathcal{P}(I - \mathcal{M}) \mathcal{M}^\nu
    = \mathcal{P}\mathcal{X}^{-1}(I - \widebar{\mathcal{M}})
        \widebar{\mathcal{M}}^\nu \mathcal{X} .
\]
Therefore, for all $\zb,\yb \in \mathbb{R}^{n+m}$ holds
\begin{align*}
  \langle \mathcal{A} \mathcal{M}^\nu \zb,\yb \rangle
   & = \langle \mathcal{P}\mathcal{X}^{-1}(I - \widebar{\mathcal{M}})
        \widebar{\mathcal{M}}^\nu \mathcal{X} \zb,\yb \rangle
     = \langle (I - \widebar{\mathcal{M}})
        \widebar{\mathcal{M}}^\nu \mathcal{X} \zb, \mathcal{X}^{-\top} \mathcal{P}^\top \mathcal{X}^{-1} \mathcal{X} \yb \rangle \\
   & \le  \|(I - \widebar{\mathcal{M}})
        \widebar{\mathcal{M}}^\nu \| \, \| \mathcal{X} \zb \| \, \|\mathcal{X}^{-\top} \mathcal{P}^\top \mathcal{X}^{-1} \| \, \|\mathcal{X} \yb\|.
\end{align*}
The rest of the proof of \eqref{firstestimate} follows directly from $\|\mathcal{K}\| = \|\mathcal{K}^\top\|$ and the estimate
\[
  \|\mathcal{X} \wb\|^2 
    =  \langle \mathcal{X}^\top\mathcal{X} \wb,\wb\rangle 
   \le \|\mathcal{X}^\top\mathcal{X}\|_{\mathcal{L} \times \mathcal{L}} \, \|\wb\|_\mathcal{L}^2.
\]
The second estimate \eqref{secondestimate} is obtained in the same way as before with the identity $\mathcal{A} \mathcal{M} = (\mathcal{P} - \mathcal{A}) (I - \mathcal{M})$.
\end{proof}

Observe that the first two factors on the right-hand sides of \eqref{firstestimate} and \eqref{secondestimate} do not dependent on the specific norm represented by $\mathcal{L}$. Only the third factor depends on the particular choice of $\mathcal{L}$. For estimating the first factor the following abstract result is helpful, see \cite[Theorem~2]{schoeberl_2003} for the proof.

\begin{theorem} \label{AR1}
Assume that a matrix $\mathcal{K}  \in \mathbb{R}^{(n+m)\times(n+m)}$ is of the form
\[
  \mathcal{K} = \mathcal{Q}^\top \mathcal{N} \mathcal{Q}
  \quad \text{with} \quad
  \mathcal{N} = \begin{bmatrix} I & - V^\top \\ V & I \end{bmatrix}^{-1},
\]
for some matrices $V \in \mathbb{R}^{m\times n}$ and $\mathcal{Q} \in \mathbb{R}^{(n+m)\times(n+m)}$ with $\|\mathcal{Q}\| \le 1$.
Then there holds
\[
  \|(I - \mathcal{K}) \mathcal{K}^\nu \| \le \eta(\nu),
\]
with
\[
  \eta(\nu) = \frac{1}{2^\nu} \binom{\nu}{ \lfloor (\nu+1)/2 \rfloor}
    \le \left\{ \begin{array}{ll}
                  \sqrt{\displaystyle
		        \frac{2}{\pi \nu}}  & \text{for even }  \nu ,\\
                  \sqrt{\displaystyle
		        \frac{2}{\pi (\nu+1)}}  & \text{for odd } \nu ,
	       \end{array} \right. 
\]
where $\binom{\nu}{\ell}$ denotes the binomial coefficient and $\lfloor x \rfloor$ denotes the largest integer smaller than or equal to $x \in \mathbb{R}$.
\end{theorem}

The function $\eta(\nu)$ represents the smoothing rate. In the following, we present sufficient conditions for the assumptions on the matrix $\mathcal{K}$, which makes the results of \cref{AR1} more directly accessible than in \cite{schoeberl_2003}.

\begin{corollary} \label{AR}
Assume that a matrix $\mathcal{K}$ is of the form
\[
  \mathcal{K} = \begin{bmatrix} M & R^\top \\ - R & N \end{bmatrix},
\]
with symmetric and positive definite matrices $M$ and $N$.
If 
\[
  \lambda_{max}(S_M) \le 1
  \quad \text{and} \quad
  \lambda_{max}(S_N) \le 1,
\]
for the primal and dual Schur complements
\[
  S_M = M + R^\top N^{-1} R
  \quad \text{and} \quad
  S_N = N + R M^{-1} R^\top,
\]
then there holds with the same $\eta$ as in \cref{AR1} the estimate
\[
  \|(I - \mathcal{K}) \mathcal{K}^\nu \| \le \eta(\nu).
\]
\end{corollary}

\begin{proof}
Observe that the inverse of $\mathcal{K}$ has the representation
\[
\mathcal{K}^{-1} = 
\begin{bmatrix} S_M^{-1} & -W^\top \\ W & S_N^{-1} \end{bmatrix}
=
\begin{bmatrix} S_M^{-1/2} & 0 \\ 0 & S_N^{-1/2} \end{bmatrix}
\begin{bmatrix} I & -V^\top \\ V & I \end{bmatrix}
\begin{bmatrix} S_M^{-1/2} & 0 \\ 0 & S_N^{-1/2} \end{bmatrix},
\]
with $W = S_N^{-1} R M^{-1} = N^{-1} R S_M^{-1}$ and $V = S_N^{-1/2} R M^{-1} S_M^{1/2} = S_N^{1/2} N^{-1} R S_M^{-1/2}$. 
Consequently, the matrix $\mathcal{K}$ has the factorization
\[
  \mathcal{K} 
    = \begin{bmatrix} S_M^{1/2} & 0 \\ 0 & S_N^{1/2} \end{bmatrix}
      \begin{bmatrix} I & - V^\top \\ V & I \end{bmatrix}^{-1}
      \begin{bmatrix} S_M^{1/2} & 0 \\ 0 & S_N^{1/2} \end{bmatrix}.
\]
The rest of the proof follows then from \cref{AR1}.
\end{proof}

In the forthcoming subsections we present the analysis of the smoothing property for the individual smoothers, which relies on the results of \cref{AR1} and \cref{AR}.

\subsection{Block diagonal preconditioners}

For the iteration matrix in the case $\mathcal{P} = \mathcal{P}_d$ we obtain the following product representation
\[
  \mathcal{M}_d
    = \mathcal{P}_d^{-1} ( \mathcal{P}_d^{} - \mathcal{A} )
    = \begin{bmatrix} 
        \hat{A}^{-1} (\hat{A} - A) & -\hat{A}^{-1} B^\top \\ 
         \hat{S}^{-1} B & \hat{S}^{-1} (\hat{S} - C) 
       \end{bmatrix} 
     = \mathcal{X}_d^{-1} \widebar{\mathcal{M}}_d^{} \mathcal{X}_d^{},
\]
with
\begin{equation} \label{Qd}
  \mathcal{X}_d^{} = \begin{bmatrix} \hat{A}^{1/2} & 0 \\ 0 & \hat{S}^{1/2} \end{bmatrix}
  \quad \text{and} \quad
  \widebar{\mathcal{M}}_d^{} =
     \begin{bmatrix} 
       I - \hat{A}^{-1/2} A \hat{A}^{-1/2} & -\hat{A}^{-1/2} B^\top \hat{S}^{-1/2} \\ 
       \hat{S}^{-1/2} B \hat{A}^{-1/2} & I - \hat{S}^{-1/2} C \hat{S}^{-1/2} 
     \end{bmatrix} .
\end{equation}
This is helpful for obtaining the following result.

\begin{theorem} \label{thrm:sp_d}
Assume that $A$ and $C$ are symmetric and positive definite. 
If $\hat{A}$ and $\hat{S}$ are symmetric and positive definite matrices satisfying
\[
  \hat{A} \ge A + B^\top C^{-1} B
  \quad \text{and} \quad 
  \hat{S} \ge C + B A^{-1} B^\top,
\]
then there holds the smoothing property
\[
  \|\mathcal{A} (\mathcal{M}_d^{})^\nu\|_{\mathcal{L} \times \mathcal{L}} 
    \le \eta(\nu) \,
    \| \mathcal{D}_d \|_{\mathcal{L} \times \mathcal{L}}
  \quad \text{with} \quad
  \mathcal{D}_d = \begin{bmatrix} \hat{A} & 0 \\ 0 & \hat{S} \end{bmatrix}.
\]
\end{theorem}

\begin{proof}
It is easy to see that the diagonal blocks of $\widebar{\mathcal{M}}_d$ are positive definite iff 
\begin{equation} \label{strict1}
   \hat{A} > A \quad \text{and} \quad \hat{S} > C.
\end{equation}
Then the primal and dual Schur complements of $\widebar{\mathcal{M}}_d$ are given by
\begin{align*}
  S_M 
   & = I - \hat{A}^{-1/2} A \hat{A}^{-1/2} + \hat{A}^{-1/2} B^\top ( \hat{S} - C )^{-1} B \hat{A}^{-1/2}, \\
  S_N 
   & = I - \hat{S}^{-1/2} C \hat{S}^{-1/2} + \hat{S}^{-1/2} B ( \hat{A} - A )^{-1} B^\top \hat{S}^{-1/2} .
\end{align*}
The condition $\lambda_\text{max}(S_M) \le 1$ of \cref{AR} leads to the inequality
$B^\top ( \hat{S} - C )^{-1} B \le A$
which is equivalent to
\begin{equation} \label{schur2}
  C + B A^{-1} B^\top \le \hat{S}.
\end{equation}
The analogous argument holds for $S_N$ and leads to the condition
\begin{equation} \label{schur1}
  A + B^\top C^{-1} B \le \hat{A}.
\end{equation}
So, if \eqref{strict1}, \eqref{schur1}, and \eqref{schur2} are satisfied, then it follows from \cref{first} and \cref{AR} that
\[
  \|\mathcal{A} (\mathcal{M^{}}_d^{})^\nu\|_{\mathcal{L} \times \mathcal{L}} 
    \le \eta(\nu) \, \|\mathcal{X}_d^{-\top} \mathcal{P}_d^{} \mathcal{X}_d^{-1}\| \, \|\mathcal{X}_d^\top \mathcal{X}_d^{}\|_{\mathcal{L} \times \mathcal{L}}.
\]
A simple closure argument implies that this inequality is also valid, if \eqref{strict1} is replaced by 
$\hat{A} \ge A$  and $\hat{S} \ge C$ which follow from \eqref{schur1}, and \eqref{schur2} anyway. Finally, we have
\[
  \mathcal{X}_d^{-\top} \mathcal{P}_d^{} \mathcal{X}_d^{-1}
    = \mathcal{J}
  \quad \text{with} \quad
  \mathcal{J}
    = \begin{bmatrix}
        I & 0 \\ 0 & -I
      \end{bmatrix}
  \quad \text{and} \quad
  \mathcal{X}_d^\top \mathcal{X}_d^{} = \mathcal{D}_d,
\]
and, therefore $\|\mathcal{X}_d^{-\top} \mathcal{P}_d^{} \mathcal{X}_d^{-1}\| = 1$, which completes the proof.
\end{proof}

\subsection{Block approximate factorization preconditioners}

The analysis of the smoothing property for the case $\mathcal{P} = \mathcal{P}_f$ was already covered in \cite{schoeberl_2003}. We briefly recall the obtained results for completeness with the notations of this paper.  From \cref{sec:classes} we have
\[
  \mathcal{P}_f 
    = \begin{bmatrix}
        \hat{A} & B^\top \\ B & B \hat{A}^{-1} B^\top - \hat{S}
      \end{bmatrix},
\]
for which we obtain the following result, see \cite{schoeberl_2003} for the proof.
\begin{theorem}\label{thrm:sp_Pf}
If $\hat{A}$ and $\hat{S}$ are symmetric and positive definite matrices satisfying
\begin{equation} \label{Condf}
  \hat{A} \ge A 
  \quad \text{and} \quad 
  \hat{S} \ge C + B \hat{A}^{-1} B^\top,
\end{equation}
then there holds the smoothing property
\[
  \|\mathcal{A} (\mathcal{M}_f^{})^\nu\|_{\mathcal{L} \times \mathcal{L}} 
    \le \eta(\nu-1) \,
    \|\mathcal{D}_f\|_{\mathcal{L} \times \mathcal{L}}
  \quad \text{with} \quad
  \mathcal{D}_f = \begin{bmatrix} \hat{A} - A & 0 \\ 0 & \hat{S} - (C + B \hat{A}^{-1} B^\top) \end{bmatrix}.
\]
\end{theorem}

\subsection{Symmetric block preconditioners}
Next we consider the new class of symmetric block matrices $\mathcal{P} = \mathcal{P}_s$. We recall from \cref{sec:classes} that in this case we assume $\hat{A} + \hat{A}^\top > A$, in order to guarantee that $\hat{A}_s = \hat{A} (\hat{A} + \hat{A}^\top - A)^{-1} \hat{A}^\top$ is well-defined. It is easy to see that the following factorization holds
\[
  \mathcal{P}_s
   = \mathcal{X}_s^\top \mathcal{J} \mathcal{X}_s^{}
  \quad \text{with} \quad
  \mathcal{J} 
    = \begin{bmatrix} I & 0 \\ 0 & - I \end{bmatrix}, \quad
  \mathcal{D}_s 
    = \begin{bmatrix} \hat{A}_s & 0  \\ 0 & \hat{S} \end{bmatrix}, \quad
  \mathcal{X}_s 
    = \mathcal{D}_s^{1/2}
      \begin{bmatrix} I & \hat{A}^{-\top} B^\top \\ 0  & I \end{bmatrix} .
\]
For the iteration matrix $\mathcal{M}_s = I - \mathcal{P}_s^{-1} \mathcal{A}$, we obtain then
\[
  \mathcal{M}_s  = \mathcal{X}_s^{-1} \widebar{\mathcal{M}}_s \, \mathcal{X}_s,
\]
with
$  \widebar{\mathcal{M}}_s 
   = \mathcal{X}_s \mathcal{M}_s \mathcal{X}_s^{-1} 
   = I - \mathcal{X}_s \mathcal{P}_s^{-1} \mathcal{A} \mathcal{X}_s^{-1}
   = I - \mathcal{J} \mathcal{X}_s^{-\top} \mathcal{A} \mathcal{X}_s^{-1}$.

\begin{theorem}\label{thrm:sp_Ps}
Assume that $A$ is symmetric and positive definite. 
If the matrix  $\hat{S}$ is symmetric and positive definite and if $\hat{A}$ and $\hat{S}$ satisfy
\begin{equation} \label{Conds}
  \hat{A} + \hat{A}^\top > A 
  \quad \text{and} \quad 
  \hat{S} \ge C + B A^{-1} B^\top,
\end{equation}
then there holds the smoothing property
\[
  \|\mathcal{A} \, (\mathcal{M}_s^{})^\nu\|_{\mathcal{L} \times \mathcal{L}} 
    \le \eta(\nu) \,
    \|\mathcal{X}_s^\top \mathcal{X}_s^{}\|_{\mathcal{L} \times \mathcal{L}}.
\]
\end{theorem}

\begin{proof}
By elementary calculations we obtain the factorization
\begin{align*}
  \widebar{\mathcal{M}}_s 
   & = \mathcal{D}_s^{-1/2}
        \begin{bmatrix} 
          \hat{A}_s - A & - (I - A \hat{A}^{-\top}) B^\top \\ 
          B (I - \hat{A}^{-1} A) & \hat{S} - C - B \hat{A}_r^{-1} B^\top
        \end{bmatrix} \mathcal{D}_s^{-1/2},
\end{align*}
with $\hat{A}_r = \hat{A}^\top (\hat{A} + \hat{A}^\top - A)^{-1} \hat{A}$. 
Elementary calculations show that the inverse of $\hat{A}_r$ has the representation
\begin{equation}\label{eq:Arinv}
\begin{aligned}
\hat{A}_r^{-1} &= (I - \hat{A}^{-1} A) (\hat{A}_s - A)^{-1} (I - A \hat{A}^{-\top})\\
 &= A^{-1} - (I - \hat{A}^{-1} A) A^{-1} (I - A \hat{A}^{-\top}).
\end{aligned}
\end{equation}
First, let us assume that
\[
  \hat{A} + \hat{A}^\top > A, \quad
  \ker (\hat{A} - A) = \{0\},
  \quad \text{and} \quad 
  \hat{S} > C + B A^{-1} B^\top.
\]
Then we conclude from the identity $I - \hat{A}_s^{-1} A = (I - \hat{A}^{-\top} A)(I - \hat{A}^{-1} A)$ that
\[
  \hat{A}_s - A = (\hat{A}^\top - A) (\hat{A}+\hat{A}^\top -A)^{-1} (\hat{A}-A) > 0,
\]
and
\[
  \hat{S} - C - B \hat{A}_r^{-1} B^\top \ge \hat{S} - C - B A^{-1} B^\top  > 0,
\]
since $\hat{A}_r = A + (\hat{A} - A) (\hat{A}+\hat{A}^\top -A)^{-1} (\hat{A}^\top-A) \ge A$ is satisfied. Therefore the diagonal blocks of $\widebar{\mathcal{M}}_s$ are positive definite.

Consequently, the primal and dual Schur complements $S_M$ and $S_N$ of $\widebar{\mathcal{M}}_s$ read
\begin{align*}
  S_M 
    & = \hat{A}_s^{-1/2} \left\{ \hat{A}_s - A 
   + (I - A \hat{A}^{-\top}) B^\top \hspace{-0.3em}\left[ \hat{S} - C - B \hat{A}_r^{-1} B^\top \right]^{-1}\hspace{-0.32em} B (I - \hat{A}^{-1} A) \right\} \hat{A}_s^{-1/2}, \\
  S_N & = \hat{S}^{-1/2} \left\{ \hat{S} - C -  B \hat{A}_r^{-1}B^\top + B (I - \hat{A}^{-1} A) (\hat{A}_s - A)^{-1} (I - A \hat{A}^{-\top}) B^\top \right\} \hat{S}^{-1/2}.
\end{align*}
The conditions $\lambda_\text{max}(S_M) \le 1$ and $\lambda_\text{max}(S_N) \le 1$ of \cref{AR} reduce to the inequalities
\[
  (I - A \hat{A}^{-\top}) B^\top \left[ \hat{S} - C - B \hat{A}_r^{-1} B^\top \right]^{-1} B (I - \hat{A}^{-1} A) \le A, 
\]
and
\[
  B (I - \hat{A}^{-1} A) (\hat{A}_s - A)^{-1} (I - A \hat{A}^{-\top}) B^\top \le  C +  B \hat{A}_r^{-1}B^\top .
\]
Due to \eqref{eq:Arinv}, the second inequality is always satisfied since it simplifies to $C \ge 0$, which trivially holds. The first inequality can be rewritten as
\[
     B^\top \left[ \hat{S} - C - B \hat{A}_r^{-1} B^\top \right]^{-1} B  \le (I - A \hat{A}^{-\top})^{-1} A (I - \hat{A}^{-1} A)^{-1},
\]
which is equivalent to
\[
  B (I - \hat{A}^{-1} A) A^{-1} (I - A \hat{A}^{-\top}) B^\top \le  \hat{S} - C - B \hat{A}_r^{-1} B^\top.
\]
Applying the identity \eqref{eq:Arinv}, this inequality simplifies to $S \le \hat{S}$. The corresponding results under the slightly weaker conditions in \eqref{Conds} follow by a closure argument.

Now we can apply \cref{AR}. Using the fact that $\mathcal{X}_s^{-\top} \mathcal{P}_s^{} \mathcal{X}_s^{-1} = \mathcal{J}$ and \cref{first} lead to the desired result.
\end{proof}

\begin{remark}
Let the assumptions of \cref{thrm:sp_Ps} be satisfied and assume additionally that, there exists a constant $c$ such that
\[
  \hat{A}_s \le c \, (\hat{A} + \hat{A}^\top - A).
\]
Then we conclude the estimate
\[
  B \hat{A}^{-1}\hat{A}_s \hat{A}^{-\top} B^\top
    =  B (\hat{A} + \hat{A}^\top - A)^{-1} B^\top
   \le c \, B \hat{A}_s^{-1} B^\top 
   \le c \, B A^{-1} B^\top 
   \le c \, S 
   \le c \, \hat{S}. 
\]
By a simple eigenvalue analysis of the generalized eigenvalue problem $\mathcal{X}_s^\top \mathcal{X}_s^{} z = \lambda \, \mathcal{D}_s z$ one obtains
\begin{align*}
  \mathcal{X}_s^\top \mathcal{X}_s^{}
   & \le \bar{c} \, \mathcal{D}_s
  \quad \text{with} \quad
  \bar{c} = 1 + c/2 + \sqrt{c^2/4 + c},
\end{align*}
which leads to the estimate
\[
  \|\mathcal{A} \, (\mathcal{M}_s^{})^\nu\|_{\mathcal{L} \times \mathcal{L}} 
    \le \bar{c} \, \eta(\nu) \,  \|\mathcal{D}_s^{}\|_{\mathcal{L} \times \mathcal{L}}
  \quad \text{with} \quad   
  \mathcal{D}_s = \begin{bmatrix} \hat{A}_s & 0  \\ 0 & \hat{S} \end{bmatrix}.
\]
\end{remark}

\subsection{Block triangular preconditioners}

Finally, we consider block triangular case $\mathcal{P} = \mathcal{P}_\ell(\hat{A},\hat{S})$, where we employ the results of the previous subsection. Recall, from \cref{sec:classes}, both $\hat{A}$ and $\hat{S}$ are assumed to be symmetric and positive definite.

It is easy to verify that for given $\hat{A}$ there exists a non-singular matrix $\hat{H}$ such that
\[
  \hat{A} = \hat{H}_s
  \quad \text{with} \quad
  \hat{H}_s = \hat{H} (\hat{H}+\hat{H}^\top - A)^{-1} \hat{H}^\top
  \quad \text{and}\quad
   \hat{H} + \hat{H}^\top > A.
\]
For example, the matrix $\hat{H} = \hat{A} + \hat{A}^{1/2}(I - \hat{A}^{-1/2}A\hat{A}^{-1/2})^{1/2} \hat{A}^{1/2}$ satisfies these conditions. Notice, $\mathcal{P}_\ell(\hat{A},\hat{S})= \mathcal{P}_\ell(\hat{H}_s,\hat{S})$ is strongly related to $\mathcal{P}_s(\hat{H}^\top,\hat{S})$, as already indicated in \cref{relation_Pl_Ps}. For the corresponding iteration matrices  $\mathcal{M}_\ell(\hat{H}_s,\hat{S})$ and $\mathcal{M}_s(\hat{H}^\top,\hat{S})$ we obtain the following result.

\begin{lemma} \label{MellMs}
Let $\nu \in \mathbb{N}$, and define for some non-singular $P \in \mathbb{R}^{n \times n}$, $R \in \mathbb{R}^{m \times m}$
\[
  \mathcal{M}_1(P) = \begin{bmatrix} I - P^{-1} A & - P^{-1} B^\top \\ 0 & I \end{bmatrix}
  \quad \text{and} \quad
  \mathcal{M}_2(R) = \begin{bmatrix} I & 0 \\ R^{-1} B & I - R^{-1} C \end{bmatrix}.
\]
Then the iteration matrices satisfy the following relation
\[
  \left(\mathcal{M}_\ell(\hat{H}_s,\hat{S})\right)^\nu = \mathcal{M}_2(\hat{S}) \mathcal{M}_1(\hat{H}^\top) \left(\mathcal{M}_s(\hat{H}^\top,\hat{S})\right)^{\nu-1}  \mathcal{M}_1(\hat{H}).
\]
\end{lemma}

\begin{proof}
For the iteration matrices we have the representations
\[
 \mathcal{M}_\ell(\hat{H}_s,\hat{S}) = \mathcal{M}_2(\hat{S}) \mathcal{M}_1(\hat{H}^\top) \mathcal{M}_1(\hat{H}),
\quad 
  \mathcal{M}_s(\hat{H}^\top,\hat{S}) = \mathcal{M}_1(\hat{H})\mathcal{M}_2(\hat{S}) \mathcal{M}_1(\hat{H}^\top).
\]
The result follows then easily by rearranging the terms.
\end{proof}

This simple lemma allows to use the already analyzed smoothing properties for the symmetric Uzawa method $\mathcal{P}_s$ to analyze the block triangular case $\mathcal{P}_\ell(\hat{A},\hat{S})$.

\begin{theorem}\label{thrm:sp_Pl} 
Assume that $A$ is symmetric and positive definite.
Let $\hat{A}$ and $\hat{S}$ be symmetric and positive definite matrices satisfying
\begin{equation} \label{Condl}
  \hat{A} \ge A 
  \quad \text{and} \quad 
  \hat{S} \ge C + B A^{-1} B^\top,
\end{equation}
then there holds the smoothing property
\[
  \|\mathcal{A} (\mathcal{M}_\ell)^\nu\|_{\mathcal{L}\times \mathcal{L}}
   \le \sqrt{2}\, \eta(\nu-1) \, \|\mathcal{D}_d \|_{\mathcal{L}\times \mathcal{L}}
   \quad \text{with} \quad  
   \mathcal{D}_d = \begin{bmatrix} \hat{A} & 0 \\ 0 & \hat{S} \end{bmatrix}.
\] 
\end{theorem}
\begin{proof}
By applying \cref{MellMs} we obtain for $\mathcal{P}_\ell = \mathcal{P}_\ell(\hat{A},\hat{S})$, $\mathcal{M}_\ell = \mathcal{M}_\ell(\hat{A},\hat{S})$ and $\mathcal{M}_s = \mathcal{M}_s(\hat{H}^\top,\hat{S})$ the identity
\[
  \mathcal{A} (\mathcal{M}_\ell)^\nu = \mathcal{P}_\ell (I - \mathcal{M}_\ell) (\mathcal{M}_\ell)^\nu \\
   = \mathcal{P}_\ell \mathcal{M}_2(\hat{S}) \mathcal{M}_1(\hat{H}^\top) \, (I - \mathcal{M}_s) (\mathcal{M}_s)^{\nu-1}  \, \mathcal{M}_1(\hat{H}).
\]
For the same similarity transformation as in the proof of \cref{thrm:sp_Ps} but with $\hat{A}$ replaced by $\hat{H}^\top$, we obtain 
\[
  \widebar{\mathcal{M}}_s
     = \widetilde{\mathcal{X}}_s^{} \mathcal{M}_s^{} \widetilde{\mathcal{X}}_s^{-1}
  \quad \text{with} \quad
  \widetilde{\mathcal{X}}_s = \widetilde{\mathcal{D}}_s^{1/2} \begin{bmatrix} I & \hat{H}^{-1} B^\top \\ 0 & I \end{bmatrix},
  \quad 
  \widetilde{\mathcal{D}}_s = \begin{bmatrix} (\hat{H}^\top)_s & 0 \\ 0 & \hat{S} \end{bmatrix} ,
\]
which implies
\[
  \mathcal{P}_\ell (I - \mathcal{M}_\ell) (\mathcal{M}_\ell)^\nu \\
    = \mathcal{V}^\top \,  
        (I -  \widebar{\mathcal{M}}_s) (\widebar{\mathcal{M}}_s)^{\nu-1}  \, 
        \mathcal{U},
\]
with
$\mathcal{U} = \widetilde{\mathcal{X}}_s \mathcal{M}_1(\hat{H})$ and $\mathcal{V}^\top = \mathcal{P}_\ell \mathcal{M}_2(\hat{S}) \mathcal{M}_1(\hat{H}^\top) \,  \widetilde{\mathcal{X}}_s^{-1}$. Hence, we conclude for all $\zb, \yb \in \mathbb{R}^{n+m}$ the estimate
\begin{align*}
  \langle \mathcal{P}_\ell (I - \mathcal{M}_\ell) (\mathcal{M}_\ell)^\nu \zb,\yb \rangle
    & = \langle  
        (I -  \widebar{\mathcal{M}}_s) (\widebar{\mathcal{M}}_s)^{\nu-1}  \, 
        \mathcal{U} \zb, \mathcal{V} \yb \rangle \\
    & \le \|(I -  \widebar{\mathcal{M}}_s) (\widebar{\mathcal{M}}_s)^{\nu-1}\|  \, 
        \|\mathcal{U} \zb \| \, \|\mathcal{V} \yb \|.
\end{align*}
In order to estimate the last two terms we consider the products. Elementary computations show that
\[
  \mathcal{U}^\top\mathcal{U} = \begin{bmatrix} \hat{A} - A & 0 \\ 0 & \hat{S} \end{bmatrix}
\quad
\text{and}
\quad
  \mathcal{V}^\top \mathcal{V}
    = \begin{bmatrix} 
          \hat{A} - A + B^\top \hat{S}^{-1} B & B^\top \hat{S}^{-1} (\hat{S} - C) \\
          (\hat{S} - C)\hat{S}^{-1} B & (\hat{S} - C)\hat{S}^{-1}(\hat{S} - C) 
        \end{bmatrix},
\]
for which the following estimates hold
\[
  \mathcal{U}^\top\mathcal{U} \le \begin{bmatrix} \hat{A} & 0 \\ 0 & \hat{S} \end{bmatrix} = \mathcal{D}_d,
\]
and
\begin{align*}
  \mathcal{V}^\top \mathcal{V}
   & = \begin{bmatrix} 
          \hat{A} - A & 0 \\
          0 & 0 
        \end{bmatrix}
      + \begin{bmatrix} B^\top & 0 \\ 0 & \hat{S} - C \end{bmatrix}
        \begin{bmatrix} \hat{S}^{-1} & \hat{S}^{-1} \\ \hat{S}^{-1} & \hat{S}^{-1} \end{bmatrix}
        \begin{bmatrix} B  & 0 \\ 0 & \hat{S} - C \end{bmatrix} \\
   & \le \begin{bmatrix} 
          \hat{A} - A & 0 \\
          0 & 0 
        \end{bmatrix}
      + 2 \begin{bmatrix} B^\top \hat{S}^{-1} B & 0 \\ 0 & (\hat{S} - C) \hat{S}^{-1}(\hat{S} - C) \end{bmatrix}
     \le \begin{bmatrix} 
          \hat{A} + A & 0 \\
          0 & 2\hat{S} 
        \end{bmatrix} 
    \le 2 \, \mathcal{D}_d.
\end{align*}
This implies
\[
  \|\mathcal{U}z\|^2 \le \langle \mathcal{D}_d \zb, \zb \rangle \le \|\mathcal{D}_d\|_{\mathcal{L}\times \mathcal{L}} \|\zb\|_\mathcal{L}^2
  \quad \text{and} \quad
  \|\mathcal{V}y\|^2 \le \langle 2 \, \mathcal{D}_d \yb, \yb\rangle \le 2 \, \|\mathcal{D}_d\|_{\mathcal{L}\times \mathcal{L}} \|\yb\|_\mathcal{L}^2.
\]
Finally,  we estimate the term $\|(I -  \widebar{\mathcal{M}}_s) (\widebar{\mathcal{M}}_s)^{\nu-1}\|$ by applying \cref{AR} which leads to
\begin{align*}
  \langle \mathcal{P}_\ell (I - \mathcal{M}_\ell) (\mathcal{M}_\ell)^\nu \zb, \yb \rangle
    & \le \sqrt{2}\, \eta(\nu-1) \|\mathcal{D}_d\|_{\mathcal{L}\times \mathcal{L}} \|\zb\|_\mathcal{L} \|\yb\|_\mathcal{L}, 
\end{align*}
and thus the proof is completed.
\end{proof}

\begin{remark}
The same smoothing property as in \cref{thrm:sp_Pl} can be shown for $\mathcal{P} = \mathcal{P}_u$. The proof is completely analogous. 
\end{remark}

\begin{remark}
Compare the conditions on $A$ of \cref{thrm:sp_Pf} with those of \cref{thrm:sp_d,thrm:sp_Ps,thrm:sp_Pl} and observe that $\mathcal{P} = \mathcal{P}_f$ is the only class, for which the smoothing property could be shown under assumptions which do not necessarily require a non-singular matrix $A$. 
\end{remark}

\section{Application to the Stokes problem and numerical results}\label{sec:numerics}

In this section, we apply the iterative method \eqref{PA} with the individual preconditioners $\mathcal{P}_\ell$, $\mathcal{P}_u$, $\mathcal{P}_f$, and $\mathcal{P}_s$ as smoothers in the context of multigrid methods with application to the Stokes problem. Numerical results demonstrate robustness and efficiency of the methods. 

Let us denote by $\Omega \subset \mathbb{R}^3$ a polyhedral domain.  For a given forcing term $f \in L^2(\Omega)^3$, the Stokes problem with homogeneous Dirichlet boundary conditions reads as
\begin{equation}\label{eq:stokes}
\begin{alignedat}{3}
-\Delta u + \nabla p &= f &\qquad& \text{in } &&\Omega,\\
\mathop{\rm\,div}u &= 0 && \text{in } &&\Omega, \\
u &= 0 && \text{on } &&\partial \Omega.
\end{alignedat}
\end{equation}
Here, $u$ denotes the fluid velocity, and $p$ the pressure. Note, the pressure is only well-defined up to an additive constant, thus we exclude those by $\langle p, 1\rangle_\Omega = 0$, which motivates the space
$L_0^2(\Omega) = \{ q \in L^2(\Omega)\, :\, \langle q, 1 \rangle_\Omega = 0 \}$ for the pressure.  Here, $\langle \cdot, \cdot \rangle_\Omega$ denotes the inner product in $L^2(\Omega)$. For the existence and uniqueness of a solution $(u,p)$ of \eqref{eq:stokes} in $V \times Q = H_0^1(\Omega)^3 \times L_0^2(\Omega)$, see, e.g., \cite{brezzi-fortin_book, elman-silvester-wathen_2005, girault-raviart_1986}.

Based on an initial triangulation, we construct a sequence of meshes $\mathcal{T} = \{\mathcal{T}_\ell\}_{\ell \geq 0}$ by successive uniform refinement, $\ell$ denotes the refinement level. For the discretization we use linear finite elements, i.e., for a given mesh $\mathcal{T}_\ell \in \mathcal{T}$ we define the function space of piecewise linear and globally continuous functions by
$S^1_\ell (\Omega) = \{ v \in \mathcal{C}(\overline{\Omega}) :  v|_T \in P_1(T),~\forall \, T \in \mathcal{T}_\ell \}$.
The discrete finite element spaces for velocity and pressure are then given by
\begin{align*}
V_\ell = [S^1_\ell(\Omega) \cap H^1_0(\Omega)]^3, \qquad Q_\ell = S^1_\ell(\Omega) \cap L^2_0(\Omega).
\end{align*}
Since this finite element pair does not satisfy the discrete inf--sup condition we need to stabilize the method. Here, we consider the standard PSPG stabilization, see, e.g., \cite{brezzi-douglas_1988}. This approach leads to finding $(u_\ell, p_\ell) \in V_\ell \times Q_\ell$ such that \eqref{discrB} is satisfied, where the bilinear and linear forms are given by
\begin{align*}
a(u,v) = \langle \nabla u, \nabla v \rangle_{\Omega}, \quad b(u, q) = -\langle \mathop{\rm div} u, q \rangle_{\Omega}, \quad f(v) = \langle f, v \rangle_{\Omega}.
\end{align*}
Furthermore, the level-dependent stabilization terms $c_\ell(\cdot,\cdot)$ and $g_\ell(\cdot)$ are given by
\begin{align*}
c_\ell(p_\ell,q_\ell) = \sum_{T \in \mathcal{T}_\ell} \delta_T\, h_{T}^2 \,\langle \nabla p_\ell, \nabla q_\ell \rangle_T,
\quad
g_\ell(q_\ell) = -\sum_{T \in \mathcal{T}_\ell} \delta_T\, h_{T}^2 \, \langle f, \nabla q_\ell \rangle_T,
\end{align*}
where $h_{T} = (\int_T  dx)^{1/3}$ denotes the local mesh size of the element $T \in \mathcal{T}_\ell$. The stabilization parameter $\delta_T > 0$ has to be chosen carefully, according to \cite{elman-silvester-wathen_2005}, $\delta_T = 1/12$ is a good choice in practice. 

\subsection{The approximation property}
As pointed out in \cref{sec:framework} we need to specify the $L^2$-like norm which depends on the Stokes problem. Therefore, let $h_\ell = \min_{T \in \mathcal{T}_\ell} h_T$ denote the mesh size of $\mathcal{T}_\ell$.
For the mesh-depending norm on $V_\ell \times Q_\ell$ we choose
\begin{equation} \label{ell-norm}
   \| (v_\ell, q_\ell) \|_{0,\ell} = \left(h_\ell^{-2}\,  \| v_\ell \|_{L^2(\Omega)}^2 + \| q_\ell \|_{L^2(\Omega)}^2 \right)^{1/2}, 
\end{equation}
and assume that
\begin{equation} \label{scaling}
  \| (v_\ell, q_\ell) \|_{0,\ell} \sim \left(\| \vb_\ell \|_{l^2}^2 + \|\qb_\ell \|_{l^2}^2\right)^{1/2},
\end{equation}
which can be always achieved by suitable scaling of the bases in $V_\ell$ and $Q_\ell$. The associated second mesh-dependent norm $\| (v_\ell, q_\ell) \|_{2,\ell}$ is then defined according to \eqref{discr2ell}.

Observe that 
\[
  (\bar{u}_\ell^{(\nu+1)},\bar{p}_\ell^{(\nu+1)}) - (u_\ell,p_\ell) = [I - P_{\ell}^{\ell-1}] [(u_\ell^{(\nu)}, p_\ell^{(\nu)}) - (u_\ell,p_\ell)],
\]
where $P_{\ell}^{\ell-1} \colon V_\ell \times Q_\ell \rightarrow V_{\ell-1} \times Q_{\ell-1}$ is defined by
\[
  \mathcal{B}_{\ell-1}(P_{\ell}^{\ell-1}(w_\ell,r_\ell),(v_{\ell-1},q_{\ell-1})) = \mathcal{B}_{\ell}((w_\ell,r_\ell),(v_{\ell-1},q_{\ell-1})),
\]
for all $(w_\ell,r_\ell) \in V_\ell \times Q_\ell$ and $(v_{\ell-1},q_{\ell-1}) \in V_{\ell-1} \times Q_{\ell-1}$. Therefore, the approximation property \eqref{approxProp} for the Stokes problem is an immediate consequence of the following result.

\begin{theorem}\label{thm:approxprop}
Assume that the solution of \eqref{eq:stokes} satisfies the regularity condition
\begin{equation} \label{regularity}
  \|u\|_{H^2(\Omega)} + \|p\|_{H^1(\Omega)} \le c_R \, \|f\|_{L^2(\Omega)} .
\end{equation}
Then we have
\[
  \|[I - P_{\ell}^{\ell-1}] (w_\ell,r_\ell)\|_{0,\ell} \le c_A \, \|(w_\ell,r_\ell)\|_{2,\ell} \quad \forall (w_\ell,r_\ell) \in V_\ell \times Q_\ell,
\]
 with a constant $c_A$ which is independent of the level $\ell$.
\end{theorem} 

For the proof of \cref{thm:approxprop} we refer to \cite{brenner:1996}. Several properties, see \cite[Lemma~5.3]{brenner:1996}, have to be shown in order to archive the estimate. In particular, the boundedness and stability of the continuous and discrete problem, inf--sup condition, discretization error estimates, and properties of the intergrid-transfer operators are required. All of these properties are well-known or follow directly from \cite{brezzi-douglas_1988}. Only the stability of the discrete problem requires some extra effort. In \cite{brezzi-douglas_1988} stability was shown with respect to a related mesh-dependent norm. Then the required stability in the standard norm follows easily by using the first step of Verf\"urth's trick \cite{verfuerth_stokes84}. We omit the details.

\subsection{The smoothing procedures}

In the following, we discuss the smoothing procedures in terms of numerical results. The matrix $A$, given by the bilinear form $a(\cdot, \cdot)$ on $V_\ell \times V_\ell$, is symmetric and positive definite. The lower triangular part, including the diagonal, of $A$ is denoted by $L$. 
Furthermore, let $M_v$ and $M_q$ denote the mass matrices on $V_\ell$ and on $Q_\ell$ representing $\langle u_\ell,v_\ell\rangle_\Omega$ and $\langle p_\ell, q_\ell\rangle_\Omega$, respectively.
For the remaining part of the section, we set
\[
  \hat{A} = L^\top, \quad \hat{A}_s = L (L + L^\top - A)^{-1} L^\top 
    \quad \text{and} \quad
  \hat{S} = \omega^{-1} \, \diag M_q,
\]
where $\omega > 0$ denotes the damping parameter.

\begin{remark}
If $\hat{A}$, $\hat{A}_s$ and $\hat{S}$ are interpreted as preconditioners of a preconditioned Richardson method, then $\hat{A}$ and $\hat{A}_s$ correspond to backward- and symmetric Gauss--Seidel method applied to $A$, respectively. Moreover, $\hat{S}$ corresponds to the $\omega$-damped Jacobi method applied to $M_q$.
\end{remark}

In order to satisfy the assumptions \eqref{Condf}, \eqref{Conds}, and \eqref{Condl} of the theoretical results obtained in \cref{sec:smoothing_property}, the damping parameter $\omega$ has to satisfy the inequality
\begin{equation}\label{omega1}
\lambda_\text{max} ((\diag M_q)^{-1} (C + B X^{-1} B^\top)) \leq \omega^{-1},
\end{equation}
where $X = \hat{A}_s$ for $\mathcal{P}_f$ and $X = A$ for $\mathcal{P}_\ell$, $\mathcal{P}_u$, $\mathcal{P}_s$.

In order to finalize the convergence analysis we have to estimate the factors $\|\mathcal{D}_d\|_{\mathcal{L}\times\mathcal{L}}$, $\|\mathcal{D}_f\|_{\mathcal{L}\times\mathcal{L}}$, $\|\mathcal{D}_s\|_{\mathcal{L}\times\mathcal{L}}$, which appeared in the analysis of the smoothing properties. Here, $\mathcal{L}$ denotes the matrix which represents the mesh-dependent norm \eqref{ell-norm}, and is given by
\[
  \mathcal{L} 
    =  \begin{bmatrix} h_\ell^{-2} \, M_v & 0 \\ 0 & M_q\end{bmatrix} .
\]
Provided $\omega$ is not too small, i.e., there is a uniform constant $c_S$ with
\begin{equation}\label{omega2}
\omega^{-1} \leq c_S \, \lambda_\text{max} ((\diag M_q)^{-1} (C + B X^{-1} B^\top)),
\end{equation}
again with $X = \hat{A}_s$ for $\mathcal{P}_f$ and $X = A$ for $\mathcal{P}_\ell$, $\mathcal{P}_u$, $\mathcal{P}_s$.
Then, we easily see by standard scaling arguments that $\|\mathcal{D}_d\|_{\mathcal{L}\times\mathcal{L}}$, $\|\mathcal{D}_f\|_{\mathcal{L}\times\mathcal{L}}$, and
$\|\mathcal{D}_s\|_{\mathcal{L}\times\mathcal{L}}$ are uniformly bounded. So, eventually, we summarize the convergence analysis in the following result.

\begin{theorem}
Assume that the solution of the Stokes problem \eqref{eq:stokes} satisfies the regularity condition \eqref{regularity}. Then the two grid method is a contraction with a uniform contraction rate for all smoothing procedures discussed above which satisfy the associated scaling conditions \eqref{omega1} and \eqref{omega2}, respectively.
\end{theorem}

In the numerical test we use \eqref{omega1} with $X = \hat{A}_s$ for all classes, in order to avoid working with the computationally expensive exact Schur complement. The damping factor $\omega$ for $\hat{S}$ is 
approximately computed by applying a few steps of the power method to the generalized eigenvalue problem 
\begin{align}\label{eq:gen_eigv}
(\diag M_q)^{-1} (C + B \hat{A}_s^{-1} B^\top) \, \boldsymbol{q} = \lambda \,  \boldsymbol{q}, 
\end{align}
on a coarse mesh. 

\subsection{Numerical results}

In the following, we present several numerical experiments for the different Uzawa variants, and investigate also the computational costs. The computational domain is $\Omega = (0,1)^3$, where the coarse grid $\mathcal{T}_0$ consists of 384 elements (a mesh of 6 elements, twice uniformly refined).
From the generalized eigenvalue problem \cref{eq:gen_eigv} we obtain numerically the value $\omega = 0.55849$.

The computations are performed on an Intel Xeon CPU E3-1226~v3, 3.30GHz with 32 GB  memory. For the experiments only a single core is used, so that the timings do not realistically represent the performance of the four-core CPU. To this end, we employ the hierarchical hybrid grids (HHG) framework \cite{bergen-gradl-ruede-huelsemann_2006, bergen-huelsemann_2004, gmeiner-huber-john-ruede-wohlmuth_2015, GRSWW2015} with the intra-node shared memory parallel execution being disabled.

\begin{remark}
The damping parameter $\omega$ and the performance results differ in the case of 
degenerated elements, in particular when the tetrahedra have obtuse angles.
\end{remark}

For all numerical results, the solution of the discretized problem is then obtained by applying multigrid iterations.
We first focus on $W$-cycles since they are covered by the theory. In the following, we denote by $\nu$ the total number of smoothing steps, where the pre- and post-smoothing steps are determined by $\nu- \lfloor\nu/2 \rfloor$ and $\lfloor \nu/2 \rfloor$, respectively. The initial guess $(\boldsymbol{u}_0, \boldsymbol{p}_0)^\top$ is a random vector with values in $[0,1]$. 

\subsubsection{Smoothing rates and asymptotic convergence rates}

In a first test, we study the smoothing property of the individual methods numerically. Therefore, we evaluate the term $\| \mathcal{A} M^\nu \|_{\mathcal{L} \times \mathcal{L}}$ for different smoothing steps $\nu \in \mathbb{N}$. Recall from \cref{AR1}, the smoothing rate $\eta(\nu) \sim \nu^{-1/2}$. The computations are performed on refinement level $\ell=4$ with approximately $10^6$ degrees of freedom. 
 
\definecolor{darkgreen}{rgb}{0.125,0.5,0.169}
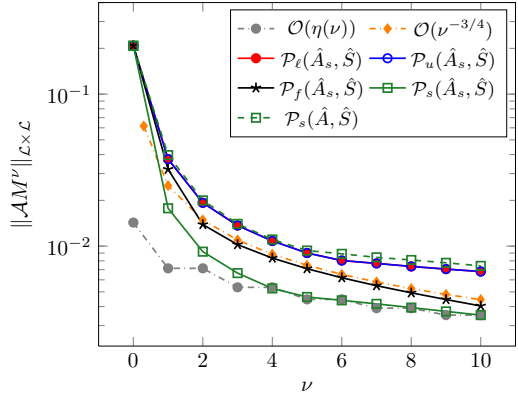
\begin{figure}[h!]
\centering
\begin{tikzpicture}[scale=0.8]
\def\scale{1.0/70.0}
\def\comma{,}
\begin{semilogyaxis}[
	legend columns=2,
	xlabel={$\nu$},
	ylabel={$\| \mathcal{A} M^\nu \|_{\mathcal{L} \times \mathcal{L}} $},
	ymin=0.0022,
	ymax=0.4
]
\addplot[gray, dashdotted, mark=*, thick, mark options={solid}] coordinates {
(0,\scale*1.)
(1, \scale*0.5)
(2, \scale*0.5)
(3, \scale*0.375)
(4, \scale*0.375)
(5, \scale*0.3125)
(6, \scale*0.3125)
(7, \scale*0.273438)
(8, \scale*0.273438)
(9, \scale*0.246094)
(10, \scale*0.246094)
};

\addplot[ orange, dashdotted, thick, domain=0:10, mark=diamond*, mark options={solid}] coordinates {
(0.3, 0.0616736)
(  1, 0.025)
(  2, 0.0148651)
(  3, 0.0109673)
(  4, 0.00883883)
(  5, 0.00747674)
(  6, 0.00652119)
(  7, 0.0058092)
(  8, 0.0052556)
(  9, 0.00481125)
(10, 0.0044457)
};

\addplot[red, mark=*, thick] coordinates {
(  0, 0.207996384120918)
(  1, 0.0373526962338543)
(  2, 0.0192745731208303)
(  3, 0.0136641073459574)
(  4, 0.0108157739843747)
(  5, 0.00900926951443624)
(  6, 0.00803610699399563)
(  7, 0.00769469817014541)
(  8, 0.0073502719714151)
(  9, 0.00706050060843181)
(10, 0.00680252365294792)
};

\addplot[blue, mark=o, thick] coordinates {
(  0, 0.207996384120918)
(  1, 0.0373520656688646)
(  2, 0.0192745676813583)
(  3, 0.0136640616589857)
(  4, 0.0108157741255856)
(  5, 0.00900926952810574)
(  6, 0.00803908022433861)
(  7, 0.00769585832630931)
(  8, 0.00735003743127909)
(  9, 0.00706049860956724)
(10, 0.00680252357903783)
};

\addplot[black, mark=star, thick, mark options = {scale=1.25}] coordinates {
(  0, 0.207996384120918)
(  1, 0.0319977254642748)
(  2, 0.0138583730332115)
(  3, 0.0102275881029119)
(  4, 0.00834330775270372)
(  5, 0.00710128361457704)
(  6, 0.00619200843079961)
(  7, 0.00549054406158394)
(  8, 0.00492885219552134)
(  9, 0.00445268882014655)
(10, 0.00403888567520626)
};

\addplot[darkgreen, mark=square, thick] coordinates {
(  0, 0.207996384120918)
(  1, 0.0177412088584518)
(  2, 0.00919509486346293)
(  3, 0.00664370743272169)
(  4, 0.00526696236126429)
(  5, 0.0046276925272888)
(  6, 0.00439708564461531)
(  7, 0.00416985959026414)
(  8, 0.00394115774957919)
(  9, 0.00372011883804373)
(10, 0.00351554551658048)
};

\addplot[darkgreen, mark=square, thick, dashed, mark options={solid}] coordinates {
(  0, 0.207996384120918)
(  1, 0.0397087540879259)
(  2, 0.019992080959832)
(  3, 0.0140072227467648)
(  4, 0.011056729520218)
(  5, 0.00935144113126181)
(  6, 0.00889666661029445)
(  7, 0.00843332168921213)
(  8, 0.00810380277189382)
(  9, 0.00776812685897727)
(10, 0.00741396308741607)
};

\legend{\small$\mathcal{O}(\eta(\nu))$, \small$\mathcal{O}(\nu^{-3/4})$, \small$\mathcal{P}_\ell(\hat{A}_s\comma \hat{S})$, \small$\mathcal{P}_u(\hat{A}_s\comma \hat{S})$, \small$\mathcal{P}_f(\hat{A}_s\comma \hat{S})$, \small$\mathcal{P}_s(\hat{A}_s\comma \hat{S})$, \small$\mathcal{P}_s(\hat{A}\comma \hat{S})$}
\end{semilogyaxis}
\end{tikzpicture}
\caption{Smoothing rates for different Uzawa variants.} 
\label{Fig:smoothingrates}
\end{figure}

In \cref{Fig:smoothingrates}, we depict the results of the numerical smoothing rates. For the cases $\mathcal{P}_\ell(\hat{A}_s, \hat{S})$, $\mathcal{P}_u(\hat{A}_s, \hat{S})$, $\mathcal{P}_s(\hat{A}_s, \hat{S})$ and $\mathcal{P}_s(\hat{A}, \hat{S})$ we observe that the smoothing rates are of order $\eta(\nu)$, as shown theoretically in \cref{sec:smoothing_property}. Notice, in the case of the smoother $\mathcal{P}_f(\hat{A}_s, \hat{S})$ a even slightly higher smoothing rate of approximately $\mathcal{O}(\nu^{-3/4})$ is observed.

Next, let us demonstrate the robustness with respect to the problem size of the considered methods. Therefore, we consider the trivial solution, i.e. homogeneous right-hand side and Dirichlet datum, with a random initial vector as described above. 

\begin{table}[h!]
\centering
\setlength{\tabcolsep}{1.5mm}
\begin{tabular}{ l l | c c c c c | c c c c c}
\hline
& & & & & & & & & & & \\[-2.5ex]
\multicolumn{2}{c|}{} & \multicolumn{5}{c|}{$\mathcal{P}_\ell(\hat{A}_s, \hat{S})$} & \multicolumn{5}{c}{$\mathcal{P}_s(\hat{A}, \hat{S})$}\\
\hline
\multicolumn{2}{c|}{} & \multicolumn{5}{c|}{$\nu$} & \multicolumn{5}{c}{$\nu$}\\
$\ell$ & \multicolumn{1}{c|}{DoFs}	& 1 & 2 & 4 & 6 & 8 & 1 & 2 & 4 & 6 & 8  \\
\hline 
& & & & & & & & & & & \\[-2.5ex]
1 &  $1.4 \cdot 10^{3}$	& 0.857		& 0.816		& 0.554		& 0.418			& 0.319	& 0.857			& 0.739			& 0.549			& 0.418			& 0.319\\	
2 &  $1.4 \cdot 10^{4}$	& 0.857		& 0.741		& 0.556		& 0.420			& 0.320	& 0.857			& 0.741			& 0.556			& 0.420			& 0.319\\	
3 &  $1.2 \cdot 10^{5}$	& 0.857		& 0.740		& 0.556		& 0.420			& 0.319	& 0.857			& 0.740			& 0.556			& 0.420			& 0.320\\	
4 &  $1.0 \cdot 10^{6}$	& 0.857		& 0.737		& 0.556		& 0.420			& 0.320	& 0.857			& 0.738			& 0.556			& 0.420			& 0.320\\	
5 &  $8.2 \cdot 10^{6}$	& 0.857		& 0.736		& 0.556		& 0.420			& 0.319	& 0.857			& 0.737			& 0.555			& 0.420			& 0.320\\	
6 &  $6.6 \cdot 10^{7}$	& 0.857		& 0.736		& 0.556		& 0.420			& 0.320	& 0.857			& 0.736			& 0.555			& 0.420			& 0.320\\	
\hline
\end{tabular}
\caption{Asymptotic convergence rates for several refinement levels in case of $\mathcal{P}_\ell(\hat{A}_s, \hat{S})$ and $\mathcal{P}_s(\hat{A}, \hat{S})$.}
\label{T:Wcycle_Pl_Ps}
\end{table}

In \cref{T:Wcycle_Pl_Ps}, we present asymptotic convergence rates for the $W$-cycles  with respect to the norm $\| \cdot \|_\mathcal{L}$ for $\mathcal{P}_\ell(\hat{A}_s, \hat{S})$ and $\mathcal{P}_s(\hat{A}, \hat{S})$ which include several smoothing steps $\nu$ and refinement levels $\ell$. We observe robustness with respect to the problem size for both cases. Also, the asymptotic convergence rates for both cases almost agree. Comparable results are obtained for the remaining variants $\mathcal{P}_i(\hat{A}_s, \hat{S})$, $i \in \{u, f, s\}$, even though $P_f(\hat{A}_s, \hat{S})$ and $P_s(\hat{A}_s, \hat{S})$ have a better pre-asymptotic behavior in the smoothing rate, cf.\ \cref{Fig:smoothingrates}. In particular, $\mathcal{P}_s(\hat{A}_s, \hat{S})$ does not improve the asymptotic convergence rate, and is thus only more costly.

\subsubsection{Relative costs and time-to-solution}

So far, we did not include the cost of the individual smoothers in the results. In the following, we evaluate the costs of the smoothers and its corresponding multigrid method. In particular, we aim to decide by means of the cost of the smoother whether or not the corresponding multigrid method is expensive. Motivated from the fact that the required number of iterations for solving up to a given tolerance $\epsilon$ is roughly given by $\log(\epsilon) / \log(\rho)$, where $\rho$ denotes the convergence rate, we introduce the relative costs for the smoothers and multigrid methods by
\begin{align}\label{rel_cost}
c_i^\text{SM}(\nu) = \frac{c_i(\nu)}{\log(\eta_i(\nu))} \frac{\log(\eta_\text{ref})}{c_\text{ref}}, \qquad c_i^\text{MG}(\nu) = \frac{\bar{c}_i(\nu)}{\log(\rho_i(\nu))} \frac{\log(\rho_\text{ref})}{\bar{c}_\text{ref}},
\end{align}
where $i \in \{\ell, u, f, s\}$ indicates the method and for some to be specified reference configuration. Here, $\eta_i(\nu)$ denotes the smoothing rate and $\rho_i(\nu)$ the multigrid convergence rate. Furthermore, $c_i(\nu)$ represents the costs of the particular smoother and $\bar{c}_i(\nu)$ additionally include the cost of the residual calculation, within the multigrid method. The costs for the individual smoothers are measured in terms of a matrix-vector multiplication with the system matrix $\mathcal{A}$ of the Stokes system and is defined by 10 scalar matrix-vector multiplications, i.e. three for $A$, $B^\top$, $B$ and one for $C$. 

In the following, we exemplarily introduce the costs for the cases $\mathcal{P}_\ell(\hat{A}_s, \hat{S})$ and $P_s(\hat{A}, \hat{S})$, which include the symmetric- and backward Gauss--Seidel, respectively. The costs for the other relevant cases can be derived similarly. Within the Gauss--Seidel scheme the application of $\hat{A}^{-1}$ and $A$ count for three matrix-vector multiplications, since we have a vector-valued velocity. The application of a damped Jacobi $\hat{S}^{-1}$ involves only the diagonal and thus has a negligible small costs. According to \cref{sec:classes} the costs are then given as follows. For $\mathcal{P}_\ell(\hat{A}_s, \hat{S})$ we obtain
\begin{align*}
c_\ell(\nu) = \nu \frac{13}{10}, \qquad \bar{c}_\ell(\nu) = \hat{n} \left( \nu \frac{13}{10} + 1\right),
\end{align*}
and for $P_s(\hat{A}, \hat{S})$ we have
\begin{align*}
c_s(\nu) = \frac{\chi(\nu)}{10},\ \ \bar{c}_s(\nu) = \hat{n} \left( \frac{\chi(\nu)}{10} + 1\right), \ \ \text{with }
\chi(\nu) = 
\begin{cases}
16 \nu \quad& \text{for } \nu \leq 2,\\
 16 \cdot 2 + 13 (\nu - 2) & \text{else}.
\end{cases}
\end{align*}
Here, $\hat{n}$ denotes the limit case of the number of all visits, weighted by $1/8$ according to the level. Since we are interested only in the relative cost, the value of $\hat{n}$ is not important.

\begin{figure}[h!]
\centering
\begin{tikzpicture}[scale=0.8]
\def\scale{1.0/70.0}
\def\comma{,}
\begin{axis}[
	xlabel={$\nu$},
	ylabel={$c_i^\text{SM}(\nu)$},
	legend pos = north west 
]

\addplot[red, mark=*, thick] coordinates {
(  1, 0.4353032634)
(  2, 0.7247433144)
(  3, 1)
(  4, 1.2644783746)
(  5, 1.5192630548)
(  6, 1.7799135227)
(  7, 2.0580438409)
(  8, 2.3301266514)
(  9, 2.6001062173)
(10, 2.8674587117)
};

\addplot[blue, mark=o, thick] coordinates {
(  1, 0.435301028)
(  2, 0.7247432626)
(  3, 0.9999992212)
(  4, 1.2644783783)
(  5, 1.5192630553)
(  6, 1.7800500265)
(  7, 2.0581075908)
(  8, 2.3301115176)
(  9, 2.6001060687)
(10, 2.8674587054)
};

\addplot[black, mark=star, thick, mark options = {scale=1.25}] coordinates {
(  1, 0.703549958)
(  2, 1.1319273749)
(  3, 1.5132714448)
(  4, 1.8858575375)
(  5, 2.24714911)
(  6, 2.5979326901)
(  7, 2.9397561538)
(  8, 3.2737121639)
(  9, 3.5985681424)
(10, 3.9143371418)
};

\addplot[darkgreen, mark=square, thick] coordinates {
(  1, 0.60063581)
(  2, 1.0329019599)
(  3, 1.383067534)
(  4, 1.7205018669)
(  5, 2.0681457378)
(  6, 2.4340578984)
(  7, 2.7921506925)
(  8, 3.1414743725)
(  9, 3.4829350061)
(10, 3.8182035477)
};

\addplot[darkgreen, mark=square, thick, dashed, mark options={solid}] coordinates {
(  1, 0.5459155213)
(  2, 0.9003246794)
(  3, 1.1605506692)
(  4, 1.4172782247)
(  5, 1.6727427527)
(  6, 1.9581259069)
(  7, 2.2358442873)
(  8, 2.5145072154)
(  9, 2.7871902702)
(10, 3.052449141)
};

\legend{\small$\mathcal{P}_\ell(\hat{A}_s\comma \hat{S})$, \small$\mathcal{P}_u(\hat{A}_s\comma \hat{S})$, \small$\mathcal{P}_f(\hat{A}_s\comma \hat{S})$, \small$\mathcal{P}_s(\hat{A}_s\comma \hat{S})$, \small$\mathcal{P}_s(\hat{A}\comma \hat{S})$}
\end{axis}
\end{tikzpicture}
\caption{Relative costs $c_i^\text{SM}(\nu)$ for different Uzawa variants.} 
\label{Fig:smoothingrates_costs}
\end{figure}
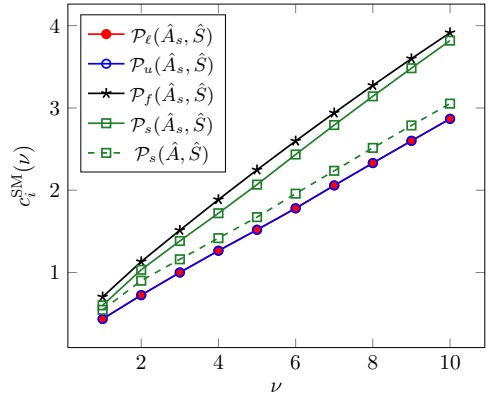

In \cref{Fig:smoothingrates_costs}, we present relative costs $c_i^\text{SM}(\nu)$ for different smoothing steps, where $\mathcal{P}_\ell(\hat{A}_s, \hat{S})$ with $\nu=3$ serves as the reference. We observe that the variants $\mathcal{P}_\ell(\hat{A}_s, \hat{S})$, $\mathcal{P}_u(\hat{A}_s, \hat{S})$ and $\mathcal{P}_s(\hat{A}, \hat{S})$ seem to incur the least cost and deliver quite similar results. 

\begin{table}[h!]
\centering
\begin{tabular}{ l | r r r r }
\hline
$\mathcal{P}\ \ \backslash\ \ \nu$	& \multicolumn{1}{c}{1} & \multicolumn{1}{c}{2} & \multicolumn{1}{c}{4} & \multicolumn{1}{c}{6} \\
\hline 
& & & & \\[-2.5ex]
$\mathcal{P}_\ell(\hat{A}_s, \hat{S})$		& 1.467 (1.504) & 1.158 (1.173) & 1.042 (1.048)	& 1.000 (1.000)\\
$\mathcal{P}_u(\hat{A}_s, \hat{S})$		& 1.463 (1.502) & 1.162 (1.176) & 1.039 (1.043)	& 1.000 (0.999)\\
$\mathcal{P}_f(\hat{A}_s, \hat{S})$		& 2.041 (2.130) & 1.737 (1.801) & 1.542 (1.602)	& 1.481 (1.535)\\
$\mathcal{P}_s(\hat{A}_s, \hat{S})$		& 2.041 (2.049) & 1.737 (1.725) & 1.532 (1.508)	& 1.475 (1.447)\\
$\mathcal{P}_s(\hat{A}, \hat{S})$		& 1.657 (1.696) & 1.352 (1.368) & 1.141 (1.145)	& 1.070 (1.070)\\
\hline
\end{tabular}
\caption{Relative costs $c_i^\text{MG}(\nu)$ for different Uzawa smoothers.} 
\label{T:relativecosts_model}
\end{table}
Next, we investigate the relative costs of the multigrid methods. For the following computation of the relative costs \eqref{rel_cost} we choose again as a reference configuration $\mathcal{P}_\ell(\hat{A}_s, \hat{S})$ with $W(3,3)$. 
Corresponding results are depicted in \cref{T:relativecosts_model}. For comparison, we also add the relative costs, where we consider instead of $\bar{c}_i$ the average time per multigrid iteration, in brackets. Again, we observe that $\mathcal{P}_\ell(\hat{A}_s, \hat{S})$, $\mathcal{P}_u(\hat{A}_s, \hat{S})$ and $\mathcal{P}_s(\hat{A}, \hat{S})$ are least costly. The numbers for the computational cost obtained by the simple model are in very good agreement with those who include the average time per multigrid iteration.
\begin{figure}[h!]
\centering
\begin{tikzpicture}[scale=0.58]
\def\scale{1.0/70.0}
\def\comma{,}
\begin{semilogyaxis}[
	xlabel={time ($\nu=1$)},
	ylabel={rel. residual},
	legend pos = north east,
	ymin = 5e-14,
	ymax = 10
]

\addplot[red, mark=*, mark repeat = 5, thick] coordinates {
(	0.00	,	1.00000E+00	)
(	3.81	,	9.60394E-02	)
(	7.62	,	1.65333E-02	)
(	11.43	,	6.95700E-03	)
(	15.24	,	2.76850E-03	)
(	19.05	,	1.34790E-03	)
(	22.86	,	7.02139E-04	)
(	26.67	,	3.35640E-04	)
(	30.48	,	1.95144E-04	)
(	34.29	,	1.17128E-04	)
(	38.11	,	7.37911E-05	)
(	41.92	,	5.05564E-05	)
(	45.73	,	3.57765E-05	)
(	49.54	,	2.62779E-05	)
(	53.35	,	1.98765E-05	)
(	57.16	,	1.52813E-05	)
(	60.97	,	1.19357E-05	)
(	64.78	,	9.43703E-06	)
(	68.59	,	7.53589E-06	)
(	72.40	,	6.06845E-06	)
(	76.21	,	4.91998E-06	)
(	80.02	,	4.01191E-06	)
(	83.83	,	3.28758E-06	)
(	87.64	,	2.70540E-06	)
(	91.45	,	2.23441E-06	)
(	95.26	,	1.85121E-06	)
(	99.07	,	1.53791E-06	)
(	102.88	,	1.28069E-06	)
(	106.69	,	1.06872E-06	)
(	110.50	,	8.93499E-07	)
(	114.32	,	7.48236E-07	)
(	118.13	,	6.27510E-07	)
(	121.94	,	5.26956E-07	)
(	125.75	,	4.43040E-07	)
(	129.56	,	3.72885E-07	)
(	133.37	,	3.14143E-07	)
(	137.18	,	2.64889E-07	)
(	140.99	,	2.23537E-07	)
(	144.80	,	1.88779E-07	)
(	148.61	,	1.59534E-07	)
(	152.42	,	1.34903E-07	)
(	156.23	,	1.14141E-07	)
(	160.04	,	9.66259E-08	)
(	163.85	,	8.18388E-08	)
(	167.66	,	6.93466E-08	)
(	171.47	,	5.87865E-08	)
(	175.28	,	4.98544E-08	)
(	179.09	,	4.22954E-08	)
(	182.90	,	3.58951E-08	)
(	186.71	,	3.04734E-08	)
(	190.53	,	2.58787E-08	)
(	194.34	,	2.19833E-08	)
(	198.15	,	1.86794E-08	)
(	201.96	,	1.58762E-08	)
(	205.77	,	1.34971E-08	)
(	209.58	,	1.14772E-08	)
(	213.39	,	9.76173E-09	)
(	217.20	,	8.30447E-09	)
(	221.01	,	7.06619E-09	)
(	224.82	,	6.01372E-09	)
(	228.63	,	5.11896E-09	)
(	232.44	,	4.35810E-09	)
(	236.25	,	3.71096E-09	)
(	240.06	,	3.16042E-09	)
(	243.87	,	2.69198E-09	)
(	247.68	,	2.29331E-09	)
(	251.49	,	1.95397E-09	)
(	255.30	,	1.66506E-09	)
(	259.11	,	1.41906E-09	)
(	262.92	,	1.20956E-09	)
(	266.74	,	1.03112E-09	)
(	270.55	,	8.79103E-10	)
(	274.36	,	7.49585E-10	)
(	278.17	,	6.39220E-10	)
(	281.98	,	5.45162E-10	)
(	285.79	,	4.64993E-10	)
(	289.60	,	3.96652E-10	)
(	293.41	,	3.38388E-10	)
(	297.22	,	2.88710E-10	)
(	301.03	,	2.46347E-10	)
(	304.84	,	2.10219E-10	)
(	308.65	,	1.79404E-10	)
(	312.46	,	1.53119E-10	)
(	316.27	,	1.30696E-10	)
(	320.08	,	1.11565E-10	)
(	323.89	,	9.52415E-11	)
(	327.70	,	8.13125E-11	)
(	331.51	,	6.94256E-11	)
(	335.32	,	5.92806E-11	)
(	339.13	,	5.06215E-11	)
(	342.95	,	4.32301E-11	)
(	346.76	,	3.69203E-11	)
(	350.57	,	3.15335E-11	)
(	354.38	,	2.69343E-11	)
(	358.19	,	2.30073E-11	)
(	362.00	,	1.96540E-11	)
(	365.81	,	1.67904E-11	)
(	369.62	,	1.43448E-11	)
(	373.43	,	1.22561E-11	)
(	377.24	,	1.04721E-11	)
(	381.05	,	8.94828E-12	)
(	384.86	,	7.64655E-12	)
(	388.67	,	6.53452E-12	)
(	392.48	,	5.58448E-12	)
(	396.29	,	4.77280E-12	)
(	400.10	,	4.07928E-12	)
(	403.91	,	3.48670E-12	)
(	407.72	,	2.98033E-12	)
(	411.53	,	2.54761E-12	)
(	415.34	,	2.17782E-12	)
(	419.16	,	1.86178E-12	)
(	422.97	,	1.59166E-12	)
(	426.78	,	1.36080E-12	)
(	430.59	,	1.16346E-12	)
(	434.40	,	9.94785E-13	)
};

\addplot[blue, mark=o, mark repeat = 5, thick] coordinates {
(	0.00	,	1.00000E+00	)
(	3.81	,	9.35015E-02	)
(	7.61	,	1.95255E-02	)
(	11.42	,	7.93694E-03	)
(	15.23	,	3.81157E-03	)
(	19.04	,	2.08466E-03	)
(	22.84	,	1.20927E-03	)
(	26.65	,	7.52612E-04	)
(	30.46	,	5.02942E-04	)
(	34.27	,	3.48656E-04	)
(	38.07	,	2.49023E-04	)
(	41.88	,	1.82117E-04	)
(	45.69	,	1.35515E-04	)
(	49.49	,	1.02317E-04	)
(	53.30	,	7.81586E-05	)
(	57.11	,	6.02812E-05	)
(	60.92	,	4.68832E-05	)
(	64.72	,	3.67248E-05	)
(	68.53	,	2.89466E-05	)
(	72.34	,	2.29397E-05	)
(	76.15	,	1.82663E-05	)
(	79.95	,	1.46068E-05	)
(	83.76	,	1.17250E-05	)
(	87.57	,	9.44409E-06	)
(	91.37	,	7.63058E-06	)
(	95.18	,	6.18282E-06	)
(	98.99	,	5.02281E-06	)
(	102.80	,	4.09025E-06	)
(	106.60	,	3.33826E-06	)
(	110.41	,	2.73017E-06	)
(	114.22	,	2.23717E-06	)
(	118.03	,	1.83653E-06	)
(	121.83	,	1.51021E-06	)
(	125.64	,	1.24386E-06	)
(	129.45	,	1.02606E-06	)
(	133.25	,	8.47607E-07	)
(	137.06	,	7.01152E-07	)
(	140.87	,	5.80758E-07	)
(	144.68	,	4.81633E-07	)
(	148.48	,	3.99900E-07	)
(	152.29	,	3.32412E-07	)
(	156.10	,	2.76612E-07	)
(	159.91	,	2.30416E-07	)
(	163.71	,	1.92125E-07	)
(	167.52	,	1.60350E-07	)
(	171.33	,	1.33951E-07	)
(	175.13	,	1.11996E-07	)
(	178.94	,	9.37178E-08	)
(	182.75	,	7.84856E-08	)
(	186.56	,	6.57799E-08	)
(	190.36	,	5.51718E-08	)
(	194.17	,	4.63074E-08	)
(	197.98	,	3.88937E-08	)
(	201.79	,	3.26885E-08	)
(	205.59	,	2.74905E-08	)
(	209.40	,	2.31332E-08	)
(	213.21	,	1.94778E-08	)
(	217.01	,	1.64092E-08	)
(	220.82	,	1.38315E-08	)
(	224.63	,	1.16648E-08	)
(	228.44	,	9.84239E-09	)
(	232.24	,	8.30868E-09	)
(	236.05	,	7.01719E-09	)
(	239.86	,	5.92907E-09	)
(	243.67	,	5.01181E-09	)
(	247.47	,	4.23820E-09	)
(	251.28	,	3.58541E-09	)
(	255.09	,	3.03432E-09	)
(	258.90	,	2.56887E-09	)
(	262.70	,	2.17558E-09	)
(	266.51	,	1.84312E-09	)
(	270.32	,	1.56198E-09	)
(	274.12	,	1.32414E-09	)
(	277.93	,	1.12285E-09	)
(	281.74	,	9.52437E-10	)
(	285.55	,	8.08114E-10	)
(	289.35	,	6.85845E-10	)
(	293.16	,	5.82227E-10	)
(	296.97	,	4.94387E-10	)
(	300.78	,	4.19901E-10	)
(	304.58	,	3.56721E-10	)
(	308.39	,	3.03115E-10	)
(	312.20	,	2.57620E-10	)
(	316.00	,	2.19000E-10	)
(	319.81	,	1.86207E-10	)
(	323.62	,	1.58355E-10	)
(	327.43	,	1.34695E-10	)
(	331.23	,	1.14590E-10	)
(	335.04	,	9.75038E-11	)
(	338.85	,	8.29792E-11	)
(	342.66	,	7.06299E-11	)
(	346.46	,	6.01280E-11	)
(	350.27	,	5.11955E-11	)
(	354.08	,	4.35965E-11	)
(	357.88	,	3.71308E-11	)
(	361.69	,	3.16285E-11	)
(	365.50	,	2.69451E-11	)
(	369.31	,	2.29583E-11	)
(	373.11	,	1.95638E-11	)
(	376.92	,	1.66733E-11	)
(	380.73	,	1.42116E-11	)
(	384.54	,	1.21147E-11	)
(	388.34	,	1.03284E-11	)
(	392.15	,	8.80642E-12	)
(	395.96	,	7.50953E-12	)
(	399.76	,	6.40429E-12	)
(	403.57	,	5.46226E-12	)
(	407.38	,	4.65925E-12	)
(	411.19	,	3.97468E-12	)
(	414.99	,	3.39099E-12	)
(	418.80	,	2.89329E-12	)
(	422.61	,	2.46884E-12	)
(	426.42	,	2.10685E-12	)
(	430.22	,	1.79808E-12	)
(	434.03	,	1.53468E-12	)
(	437.84	,	1.30997E-12	)
(	441.64	,	1.11825E-12	)
(	445.45	,	9.54665E-13	)
};

\addplot[black, mark=star, mark repeat = 5, thick, mark options = {scale=1.25}] coordinates {
(	0.00	,	1.00000E+00	)
(	5.26	,	2.11163E-02	)
(	10.53	,	4.17874E-03	)
(	15.79	,	1.20315E-03	)
(	21.05	,	4.39688E-04	)
(	26.32	,	2.12610E-04	)
(	31.58	,	1.22314E-04	)
(	36.84	,	8.17054E-05	)
(	42.11	,	5.81348E-05	)
(	47.37	,	4.31825E-05	)
(	52.63	,	3.30484E-05	)
(	57.90	,	2.58440E-05	)
(	63.16	,	2.05352E-05	)
(	68.42	,	1.65150E-05	)
(	73.69	,	1.34059E-05	)
(	78.95	,	1.09615E-05	)
(	84.21	,	9.01447E-06	)
(	89.47	,	7.44773E-06	)
(	94.74	,	6.17662E-06	)
(	100.00	,	5.13852E-06	)
(	105.26	,	4.28612E-06	)
(	110.53	,	3.58307E-06	)
(	115.79	,	3.00105E-06	)
(	121.05	,	2.51771E-06	)
(	126.32	,	2.11524E-06	)
(	131.58	,	1.77936E-06	)
(	136.84	,	1.49848E-06	)
(	142.11	,	1.26319E-06	)
(	147.37	,	1.06579E-06	)
(	152.63	,	8.99968E-07	)
(	157.90	,	7.60495E-07	)
(	163.16	,	6.43062E-07	)
(	168.42	,	5.44091E-07	)
(	173.69	,	4.60608E-07	)
(	178.95	,	3.90132E-07	)
(	184.21	,	3.30596E-07	)
(	189.48	,	2.80267E-07	)
(	194.74	,	2.37697E-07	)
(	200.00	,	2.01669E-07	)
(	205.27	,	1.71162E-07	)
(	210.53	,	1.45318E-07	)
(	215.79	,	1.23414E-07	)
(	221.06	,	1.04843E-07	)
(	226.32	,	8.90906E-08	)
(	231.58	,	7.57247E-08	)
(	236.84	,	6.43797E-08	)
(	242.11	,	5.47472E-08	)
(	247.37	,	4.65662E-08	)
(	252.63	,	3.96159E-08	)
(	257.90	,	3.37098E-08	)
(	263.16	,	2.86896E-08	)
(	268.42	,	2.44214E-08	)
(	273.69	,	2.07919E-08	)
(	278.95	,	1.77047E-08	)
(	284.21	,	1.50782E-08	)
(	289.48	,	1.28434E-08	)
(	294.74	,	1.09414E-08	)
(	300.00	,	9.32233E-09	)
(	305.27	,	7.94394E-09	)
(	310.53	,	6.77023E-09	)
(	315.79	,	5.77065E-09	)
(	321.06	,	4.91925E-09	)
(	326.32	,	4.19395E-09	)
(	331.58	,	3.57599E-09	)
(	336.85	,	3.04941E-09	)
(	342.11	,	2.60064E-09	)
(	347.37	,	2.21814E-09	)
(	352.64	,	1.89208E-09	)
(	357.90	,	1.61411E-09	)
(	363.16	,	1.37710E-09	)
(	368.43	,	1.17500E-09	)
(	373.69	,	1.00264E-09	)
(	378.95	,	8.55638E-10	)
(	384.22	,	7.30248E-10	)
(	389.48	,	6.23282E-10	)
(	394.74	,	5.32026E-10	)
(	400.00	,	4.54165E-10	)
(	405.27	,	3.87727E-10	)
(	410.53	,	3.31031E-10	)
(	415.79	,	2.82646E-10	)
(	421.06	,	2.41348E-10	)
(	426.32	,	2.06099E-10	)
(	431.58	,	1.76009E-10	)
(	436.85	,	1.50321E-10	)
(	442.11	,	1.28390E-10	)
(	447.37	,	1.09666E-10	)
(	452.64	,	9.36772E-11	)
(	457.90	,	8.00243E-11	)
(	463.16	,	6.83651E-11	)
(	468.43	,	5.84077E-11	)
(	473.69	,	4.99033E-11	)
(	478.95	,	4.26393E-11	)
(	484.22	,	3.64346E-11	)
(	489.48	,	3.11343E-11	)
(	494.74	,	2.66064E-11	)
(	500.01	,	2.27380E-11	)
(	505.27	,	1.94330E-11	)
(	510.53	,	1.66091E-11	)
(	515.80	,	1.41962E-11	)
(	521.06	,	1.21344E-11	)
(	526.32	,	1.03725E-11	)
(	531.59	,	8.86679E-12	)
(	536.85	,	7.57997E-12	)
(	542.11	,	6.48017E-12	)
(	547.37	,	5.54016E-12	)
(	552.64	,	4.73669E-12	)
(	557.90	,	4.04991E-12	)
(	563.16	,	3.46283E-12	)
(	568.43	,	2.96097E-12	)
(	573.69	,	2.53193E-12	)
(	578.95	,	2.16514E-12	)
(	584.22	,	1.85155E-12	)
(	589.48	,	1.58343E-12	)
(	594.74	,	1.35418E-12	)
(	600.01	,	1.15817E-12	)
(	605.27	,	9.90555E-13	)
};

\addplot[darkgreen, mark=square, mark repeat = 5, thick] coordinates {
(	0.00	,	1.00000E+00	)
(	5.21	,	2.11163E-02	)
(	10.41	,	4.17874E-03	)
(	15.62	,	1.20315E-03	)
(	20.82	,	4.39688E-04	)
(	26.03	,	2.12610E-04	)
(	31.23	,	1.22314E-04	)
(	36.44	,	8.17054E-05	)
(	41.64	,	5.81348E-05	)
(	46.85	,	4.31825E-05	)
(	52.06	,	3.30484E-05	)
(	57.26	,	2.58440E-05	)
(	62.47	,	2.05352E-05	)
(	67.67	,	1.65150E-05	)
(	72.88	,	1.34059E-05	)
(	78.08	,	1.09615E-05	)
(	83.29	,	9.01447E-06	)
(	88.49	,	7.44773E-06	)
(	93.70	,	6.17662E-06	)
(	98.90	,	5.13852E-06	)
(	104.11	,	4.28612E-06	)
(	109.32	,	3.58307E-06	)
(	114.52	,	3.00105E-06	)
(	119.73	,	2.51771E-06	)
(	124.93	,	2.11524E-06	)
(	130.14	,	1.77936E-06	)
(	135.34	,	1.49848E-06	)
(	140.55	,	1.26319E-06	)
(	145.75	,	1.06579E-06	)
(	150.96	,	8.99968E-07	)
(	156.17	,	7.60495E-07	)
(	161.37	,	6.43062E-07	)
(	166.58	,	5.44091E-07	)
(	171.78	,	4.60608E-07	)
(	176.99	,	3.90132E-07	)
(	182.19	,	3.30596E-07	)
(	187.40	,	2.80267E-07	)
(	192.60	,	2.37697E-07	)
(	197.81	,	2.01669E-07	)
(	203.01	,	1.71162E-07	)
(	208.22	,	1.45318E-07	)
(	213.43	,	1.23414E-07	)
(	218.63	,	1.04843E-07	)
(	223.84	,	8.90906E-08	)
(	229.04	,	7.57247E-08	)
(	234.25	,	6.43797E-08	)
(	239.45	,	5.47472E-08	)
(	244.66	,	4.65662E-08	)
(	249.86	,	3.96159E-08	)
(	255.07	,	3.37098E-08	)
(	260.28	,	2.86896E-08	)
(	265.48	,	2.44214E-08	)
(	270.69	,	2.07919E-08	)
(	275.89	,	1.77047E-08	)
(	281.10	,	1.50782E-08	)
(	286.30	,	1.28434E-08	)
(	291.51	,	1.09414E-08	)
(	296.71	,	9.32233E-09	)
(	301.92	,	7.94394E-09	)
(	307.13	,	6.77023E-09	)
(	312.33	,	5.77065E-09	)
(	317.54	,	4.91925E-09	)
(	322.74	,	4.19395E-09	)
(	327.95	,	3.57599E-09	)
(	333.15	,	3.04941E-09	)
(	338.36	,	2.60064E-09	)
(	343.56	,	2.21814E-09	)
(	348.77	,	1.89208E-09	)
(	353.97	,	1.61411E-09	)
(	359.18	,	1.37710E-09	)
(	364.39	,	1.17500E-09	)
(	369.59	,	1.00264E-09	)
(	374.80	,	8.55638E-10	)
(	380.00	,	7.30248E-10	)
(	385.21	,	6.23282E-10	)
(	390.41	,	5.32026E-10	)
(	395.62	,	4.54165E-10	)
(	400.82	,	3.87727E-10	)
(	406.03	,	3.31031E-10	)
(	411.24	,	2.82646E-10	)
(	416.44	,	2.41348E-10	)
(	421.65	,	2.06099E-10	)
(	426.85	,	1.76009E-10	)
(	432.06	,	1.50321E-10	)
(	437.26	,	1.28390E-10	)
(	442.47	,	1.09666E-10	)
(	447.67	,	9.36772E-11	)
(	452.88	,	8.00243E-11	)
(	458.08	,	6.83651E-11	)
(	463.29	,	5.84077E-11	)
(	468.50	,	4.99033E-11	)
(	473.70	,	4.26393E-11	)
(	478.91	,	3.64346E-11	)
(	484.11	,	3.11343E-11	)
(	489.32	,	2.66064E-11	)
(	494.52	,	2.27380E-11	)
(	499.73	,	1.94330E-11	)
(	504.93	,	1.66091E-11	)
(	510.14	,	1.41962E-11	)
(	515.35	,	1.21344E-11	)
(	520.55	,	1.03725E-11	)
(	525.76	,	8.86679E-12	)
(	530.96	,	7.57997E-12	)
(	536.17	,	6.48017E-12	)
(	541.37	,	5.54016E-12	)
(	546.58	,	4.73669E-12	)
(	551.78	,	4.04991E-12	)
(	556.99	,	3.46283E-12	)
(	562.20	,	2.96097E-12	)
(	567.40	,	2.53193E-12	)
(	572.61	,	2.16514E-12	)
(	577.81	,	1.85155E-12	)
(	583.02	,	1.58343E-12	)
(	588.22	,	1.35418E-12	)
(	593.43	,	1.15817E-12	)
(	598.63	,	9.90555E-13	)
};

\addplot[darkgreen, mark=square, mark repeat = 5, thick, dashed, mark options={solid}] coordinates {
(	0.00	,	1.00000E+00	)
(	4.30	,	9.55771E-02	)
(	8.61	,	1.58467E-02	)
(	12.91	,	7.21370E-03	)
(	17.21	,	3.19954E-03	)
(	21.52	,	1.39366E-03	)
(	25.82	,	7.73795E-04	)
(	30.12	,	4.01353E-04	)
(	34.43	,	2.28440E-04	)
(	38.73	,	1.47574E-04	)
(	43.03	,	9.77104E-05	)
(	47.34	,	6.84828E-05	)
(	51.64	,	5.00584E-05	)
(	55.94	,	3.73712E-05	)
(	60.25	,	2.85040E-05	)
(	64.55	,	2.20731E-05	)
(	68.86	,	1.72737E-05	)
(	73.16	,	1.36463E-05	)
(	77.46	,	1.08656E-05	)
(	81.77	,	8.70937E-06	)
(	86.07	,	7.02169E-06	)
(	90.37	,	5.68940E-06	)
(	94.68	,	4.63015E-06	)
(	98.98	,	3.78284E-06	)
(	103.28	,	3.10137E-06	)
(	107.59	,	2.55062E-06	)
(	111.89	,	2.10358E-06	)
(	116.19	,	1.73930E-06	)
(	120.50	,	1.44141E-06	)
(	124.80	,	1.19704E-06	)
(	129.10	,	9.96006E-07	)
(	133.41	,	8.30180E-07	)
(	137.71	,	6.93072E-07	)
(	142.01	,	5.79459E-07	)
(	146.32	,	4.85127E-07	)
(	150.62	,	4.06658E-07	)
(	154.92	,	3.41276E-07	)
(	159.23	,	2.86711E-07	)
(	163.53	,	2.41109E-07	)
(	167.83	,	2.02946E-07	)
(	172.14	,	1.70970E-07	)
(	176.44	,	1.44146E-07	)
(	180.74	,	1.21620E-07	)
(	185.05	,	1.02686E-07	)
(	189.35	,	8.67553E-08	)
(	193.66	,	7.33403E-08	)
(	197.96	,	6.20347E-08	)
(	202.26	,	5.24998E-08	)
(	206.57	,	4.44525E-08	)
(	210.87	,	3.76563E-08	)
(	215.17	,	3.19132E-08	)
(	219.48	,	2.70572E-08	)
(	223.78	,	2.29490E-08	)
(	228.08	,	1.94718E-08	)
(	232.39	,	1.65272E-08	)
(	236.69	,	1.40325E-08	)
(	240.99	,	1.19181E-08	)
(	245.30	,	1.01253E-08	)
(	249.60	,	8.60454E-09	)
(	253.90	,	7.31415E-09	)
(	258.21	,	6.21885E-09	)
(	262.51	,	5.28884E-09	)
(	266.81	,	4.49893E-09	)
(	271.12	,	3.82784E-09	)
(	275.42	,	3.25752E-09	)
(	279.72	,	2.77273E-09	)
(	284.03	,	2.36053E-09	)
(	288.33	,	2.00998E-09	)
(	292.63	,	1.71178E-09	)
(	296.94	,	1.45806E-09	)
(	301.24	,	1.24214E-09	)
(	305.54	,	1.05836E-09	)
(	309.85	,	9.01905E-10	)
(	314.15	,	7.68684E-10	)
(	318.46	,	6.55230E-10	)
(	322.76	,	5.58593E-10	)
(	327.06	,	4.76269E-10	)
(	331.37	,	4.06126E-10	)
(	335.67	,	3.46353E-10	)
(	339.97	,	2.95411E-10	)
(	344.28	,	2.51988E-10	)
(	348.58	,	2.14971E-10	)
(	352.88	,	1.83409E-10	)
(	357.19	,	1.56497E-10	)
(	361.49	,	1.33546E-10	)
(	365.79	,	1.13972E-10	)
(	370.10	,	9.72749E-11	)
(	374.40	,	8.30312E-11	)
(	378.70	,	7.08790E-11	)
(	383.01	,	6.05102E-11	)
(	387.31	,	5.16623E-11	)
(	391.61	,	4.41114E-11	)
(	395.92	,	3.76670E-11	)
(	400.22	,	3.21663E-11	)
(	404.52	,	2.74708E-11	)
(	408.83	,	2.34624E-11	)
(	413.13	,	2.00401E-11	)
(	417.43	,	1.71181E-11	)
(	421.74	,	1.46231E-11	)
(	426.04	,	1.24925E-11	)
(	430.35	,	1.06729E-11	)
(	434.65	,	9.11893E-12	)
(	438.95	,	7.79162E-12	)
(	443.26	,	6.65788E-12	)
(	447.56	,	5.68941E-12	)
(	451.86	,	4.86207E-12	)
(	456.17	,	4.15526E-12	)
(	460.47	,	3.55137E-12	)
(	464.77	,	3.03539E-12	)
(	469.08	,	2.59450E-12	)
(	473.38	,	2.21776E-12	)
(	477.68	,	1.89581E-12	)
(	481.99	,	1.62066E-12	)
(	486.29	,	1.38551E-12	)
(	490.59	,	1.18453E-12	)
(	494.90	,	1.01275E-12	)
(	499.20	,	8.65913E-13	)
};

\end{semilogyaxis}
\end{tikzpicture}
\hspace*{-0.3cm}
\begin{tikzpicture}[scale=0.58]
\def\scale{1.0/70.0}
\def\comma{,}
\begin{semilogyaxis}[
	xlabel={time ($\nu=6$)},
	legend pos = north east,
	yticklabels={,},
	ymin = 5e-14,
	ymax = 10
]

\addplot[red, mark=*, thick] coordinates {
(	0.00	,	1.00000E+00	)
(	14.35	,	1.92213E-03	)
(	28.70	,	1.28204E-04	)
(	43.05	,	1.91986E-05	)
(	57.41	,	5.67390E-06	)
(	71.76	,	1.97817E-06	)
(	86.11	,	7.18092E-07	)
(	100.46	,	2.65953E-07	)
(	114.81	,	9.97287E-08	)
(	129.16	,	3.77293E-08	)
(	143.51	,	1.43678E-08	)
(	157.86	,	5.49970E-09	)
(	172.22	,	2.11390E-09	)
(	186.57	,	8.15300E-10	)
(	200.92	,	3.15354E-10	)
(	215.27	,	1.22278E-10	)
(	229.62	,	4.75141E-11	)
(	243.97	,	1.84973E-11	)
(	258.32	,	7.21292E-12	)
(	272.67	,	2.81678E-12	)
(	287.03	,	1.10146E-12	)
(	301.38	,	4.31226E-13	)
};

\addplot[blue, mark=o, thick] coordinates {
(	0.00	,	1.00000E+00	)
(	14.32	,	2.42791E-03	)
(	28.64	,	1.98639E-04	)
(	42.96	,	3.61407E-05	)
(	57.28	,	1.00590E-05	)
(	71.60	,	3.23734E-06	)
(	85.92	,	1.11082E-06	)
(	100.24	,	3.95897E-07	)
(	114.56	,	1.44562E-07	)
(	128.88	,	5.36753E-08	)
(	143.20	,	2.01663E-08	)
(	157.52	,	7.64303E-09	)
(	171.84	,	2.91594E-09	)
(	186.16	,	1.11824E-09	)
(	200.48	,	4.30613E-10	)
(	214.80	,	1.66381E-10	)
(	229.12	,	6.44677E-11	)
(	243.44	,	2.50388E-11	)
(	257.76	,	9.74479E-12	)
(	272.08	,	3.79930E-12	)
(	286.40	,	1.48359E-12	)
(	300.72	,	5.80134E-13	)
};

\addplot[black, mark=star, thick, mark options = {scale=1.25}] coordinates {
(	0.00	,	1.00000E+00	)
(	22.71	,	4.24715E-04	)
(	45.41	,	2.87254E-05	)
(	68.12	,	9.00984E-06	)
(	90.83	,	3.21825E-06	)
(	113.53	,	1.18561E-06	)
(	136.24	,	4.43548E-07	)
(	158.95	,	1.67589E-07	)
(	181.65	,	6.37686E-08	)
(	204.36	,	2.43945E-08	)
(	227.07	,	9.37178E-09	)
(	249.77	,	3.61295E-09	)
(	272.48	,	1.39691E-09	)
(	295.19	,	5.41440E-10	)
(	317.89	,	2.10314E-10	)
(	340.60	,	8.18469E-11	)
(	363.31	,	3.19051E-11	)
(	386.01	,	1.24555E-11	)
(	408.72	,	4.86903E-12	)
(	431.43	,	1.90567E-12	)
(	454.13	,	7.46670E-13	)
};

\addplot[darkgreen, mark=square, thick] coordinates {
(	0.00	,	1.00000E+00	)
(	22.69	,	4.24715E-04	)
(	45.38	,	2.87254E-05	)
(	68.08	,	9.00984E-06	)
(	90.77	,	3.21825E-06	)
(	113.46	,	1.18561E-06	)
(	136.15	,	4.43548E-07	)
(	158.85	,	1.67589E-07	)
(	181.54	,	6.37686E-08	)
(	204.23	,	2.43945E-08	)
(	226.92	,	9.37178E-09	)
(	249.62	,	3.61295E-09	)
(	272.31	,	1.39691E-09	)
(	295.00	,	5.41440E-10	)
(	317.69	,	2.10314E-10	)
(	340.39	,	8.18469E-11	)
(	363.08	,	3.19051E-11	)
(	385.77	,	1.24555E-11	)
(	408.46	,	4.86903E-12	)
(	431.16	,	1.90567E-12	)
(	453.85	,	7.46670E-13	)
};

\addplot[darkgreen, mark=square, thick, dashed, mark options={solid}] coordinates {
(	0.00	,	1.00000E+00	)
(	15.30	,	1.84489E-03	)
(	30.60	,	1.12983E-04	)
(	45.89	,	1.92836E-05	)
(	61.19	,	6.02693E-06	)
(	76.49	,	2.11700E-06	)
(	91.79	,	7.70640E-07	)
(	107.09	,	2.86208E-07	)
(	122.39	,	1.07587E-07	)
(	137.68	,	4.07947E-08	)
(	152.98	,	1.55662E-08	)
(	168.28	,	5.96903E-09	)
(	183.58	,	2.29789E-09	)
(	198.88	,	8.87485E-10	)
(	214.17	,	3.43694E-10	)
(	229.47	,	1.33411E-10	)
(	244.77	,	5.18904E-11	)
(	260.07	,	2.02185E-11	)
(	275.37	,	7.89019E-12	)
(	290.67	,	3.08342E-12	)
(	305.96	,	1.20649E-12	)
(	321.26	,	4.72616E-13	)
};

\end{semilogyaxis}
\end{tikzpicture}
\hspace*{-0.3cm}
\begin{tikzpicture}[scale=0.58]
\def\scale{1.0/70.0}
\def\comma{,}
\begin{semilogyaxis}[
	xlabel={time ($\nu=10$)},
	legend pos = north east,
	yticklabels={,},
	ymin = 5e-14,
	ymax = 10
]

\addplot[red, mark=*, thick] coordinates {
(	0.00	,	1.00000E+00	)
(	22.70	,	5.95386E-04	)
(	45.39	,	1.86398E-05	)
(	68.09	,	2.04370E-06	)
(	90.78	,	3.76969E-07	)
(	113.48	,	7.41588E-08	)
(	136.17	,	1.48985E-08	)
(	158.87	,	3.03157E-09	)
(	181.56	,	6.22331E-10	)
(	204.26	,	1.28619E-10	)
(	226.95	,	2.67221E-11	)
(	249.65	,	5.57581E-12	)
(	272.34	,	1.16760E-12	)
(	295.04	,	2.45244E-13	)
};

\addplot[blue, mark=o, thick] coordinates {
(	0.00	,	1.00000E+00	)
(	22.70	,	6.34110E-04	)
(	45.39	,	2.46375E-05	)
(	68.09	,	2.73197E-06	)
(	90.79	,	4.42521E-07	)
(	113.48	,	8.03582E-08	)
(	136.18	,	1.54115E-08	)
(	158.88	,	3.04924E-09	)
(	181.57	,	6.14963E-10	)
(	204.27	,	1.25614E-10	)
(	226.97	,	2.58872E-11	)
(	249.66	,	5.37028E-12	)
(	272.36	,	1.11971E-12	)
(	295.06	,	2.34410E-13	)
};

\addplot[black, mark=star, thick, mark options = {scale=1.25}] coordinates {
(	0.00	,	1.00000E+00	)
(	36.80	,	1.31230E-04	)
(	73.59	,	5.40041E-06	)
(	110.39	,	9.79168E-07	)
(	147.18	,	1.94933E-07	)
(	183.98	,	3.95383E-08	)
(	220.77	,	8.10478E-09	)
(	257.57	,	1.67370E-09	)
(	294.36	,	3.47569E-10	)
(	331.16	,	7.25007E-11	)
(	367.95	,	1.51787E-11	)
(	404.75	,	3.18765E-12	)
(	441.54	,	6.71203E-13	)
};

\addplot[darkgreen, mark=square, thick] coordinates {
(	0.00	,	1.00000E+00	)
(	36.67	,	1.31230E-04	)
(	73.34	,	5.40041E-06	)
(	110.00	,	9.79168E-07	)
(	146.67	,	1.94933E-07	)
(	183.34	,	3.95383E-08	)
(	220.01	,	8.10478E-09	)
(	256.67	,	1.67370E-09	)
(	293.34	,	3.47569E-10	)
(	330.01	,	7.25007E-11	)
(	366.68	,	1.51787E-11	)
(	403.34	,	3.18765E-12	)
(	440.01	,	6.71203E-13	)
};

\addplot[darkgreen, mark=square, thick, dashed, mark options={solid}] coordinates {
(	0.00	,	1.00000E+00	)
(	23.74	,	6.08236E-04	)
(	47.48	,	1.84802E-05	)
(	71.22	,	1.96210E-06	)
(	94.96	,	3.60286E-07	)
(	118.70	,	7.09713E-08	)
(	142.44	,	1.42862E-08	)
(	166.18	,	2.91340E-09	)
(	189.92	,	5.99219E-10	)
(	213.66	,	1.24056E-10	)
(	237.40	,	2.58114E-11	)
(	261.14	,	5.39260E-12	)
(	284.88	,	1.13045E-12	)
(	308.62	,	2.37664E-13	)
};

\legend{\small$\mathcal{P}_\ell(\hat{A}_s\comma \hat{S})$, \small$\mathcal{P}_u(\hat{A}_s\comma \hat{S})$, \small$\mathcal{P}_f(\hat{A}_s\comma \hat{S})$, \small$\mathcal{P}_s(\hat{A}_s\comma \hat{S})$, \small$\mathcal{P}_s(\hat{A}\comma \hat{S})$}
\end{semilogyaxis}
\end{tikzpicture}
\caption{Relative residual for different Uzawa variants and smoothing steps.}  
\label{Fig:residual_time}
\end{figure}
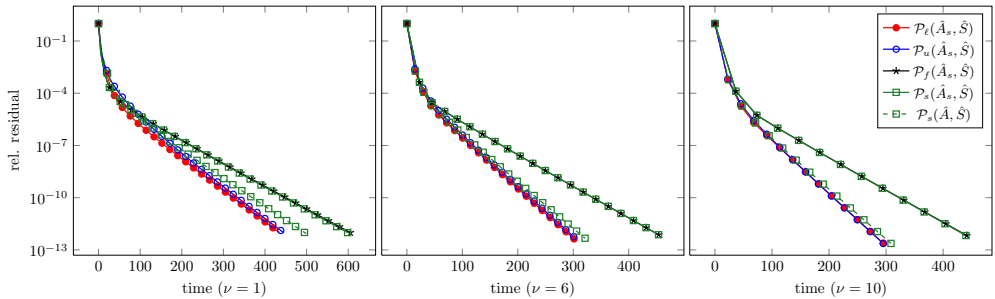

For a better understanding of the different Uzawa variants we present \cref{Fig:residual_time}, which shows the relative residual over time. The residual is measured here with respect to $\mathcal{L}^{-1}$-norm. Notice, the residual in  $\mathcal{L}^{-1}$-norm is equal to the error in the norm $\| \cdot\|_{2,\ell}$. We observe, that the methods are less different in case of fewer smoothing steps ($\nu=1$), while for larger numbers of smoothing steps the methods $\mathcal{P}_\ell(\hat{A}_s, \hat{S})$, $\mathcal{P}_u(\hat{A}_s, \hat{S})$, and $\mathcal{P}_s(\hat{A}, \hat{S})$ seem to be of similar behavior and less time consuming than $\mathcal{P}_f(\hat{A}_s, \hat{S})$ and $\mathcal{P}_s(\hat{A}_s, \hat{S})$.

Finally, we present numerical results which highlight the iteration numbers and time-to-solution. As we have seen the Uzawa variant $\mathcal{P}_\ell$ seems to be one of the most attractive methods with respect to computational cost, it is thus chosen in the forthcoming experiments. Multigrid iterations are performed until the residual in the $\mathcal{L}^{-1}$-norm is reduced by a factor $\epsilon = 10^{-8}$. 

Note that the timings shown are not representative for the HHG software when used in parallel,
since the experiments are conducted to permit a  comparison assuming sequential single core execution.

So far, we consider for the pressure smoother $\hat{S}$ only a damped Jacobi method. As in \cite{gmeiner-huber-john-ruede-wohlmuth_2015}, we also use a damped (forward) Gauss--Seidel method and damped symmetric Gauss--Seidel method, which are both applied to $C$, the associated matrices are denoted by $\widetilde{S}$ and $\widetilde{S}_s$, respectively. The damping parameters are given by $\omega(\widetilde{S}) = 0.3$ and $\omega(\widetilde{S}_s) = 0.23$.
Note, the case including $\widetilde{S}$ is not covered by the theory presented in this article.

\begin{table}[h!]
\centering
\setlength{\tabcolsep}{0.55mm}
\begin{tabular}{ l l | r r r r r r r r r r r r}
\hline
\multicolumn{2}{c|}{} & \multicolumn{12}{c}{} \\[-2.5ex]
\multicolumn{2}{c|}{} & \multicolumn{12}{c}{$\mathcal{P}_\ell(\hat{A}_s, \hat{S})$} \\
\hline
\multicolumn{2}{c|}{$\nu$} & \multicolumn{2}{c}{1}  & \multicolumn{2}{c}{2} & \multicolumn{2}{c}{4} & \multicolumn{2}{c}{6}  & \multicolumn{2}{c}{8} & \multicolumn{2}{c}{10}\\
$\ell$ & \multicolumn{1}{c|}{DoFs}	& \multicolumn{1}{c}{iter} & \multicolumn{1}{c}{time} & \multicolumn{1}{c}{iter} & \multicolumn{1}{c}{time} & \multicolumn{1}{c}{iter} & \multicolumn{1}{c}{time} & \multicolumn{1}{c}{iter} & \multicolumn{1}{c}{time} & \multicolumn{1}{c}{iter} & \multicolumn{1}{c}{time} & \multicolumn{1}{c}{iter} & \multicolumn{1}{c}{time}\\
\hline 
& & & & & & & & & & & \\[-2.5ex]
1 &  $1.4 \cdot 10^{3}$	& 66 & 0.44	& 66 & 0.46	& 17 & 0.13	& 12 & 0.10	& 9 & 0.08	& 8 & 0.07	\\
2 &  $1.4 \cdot 10^{4}$	& 64 & 0.93	& 32 & 0.52	& 17 & 0.31	& 12 & 0.25	& 9 & 0.21	& 7 & 0.18	\\
3 &  $1.2 \cdot 10^{5}$	& 62 & 2.26	& 31 & 1.33	& 17 & 0.92	& 12 & 0.81	& 9 & 0.71	& 7 & 0.64	\\
4 &  $1.0 \cdot 10^{6}$	& 60 & 7.20	& 30 & 4.81	& 16 & 3.79	& 12 & 3.80	& 9 & 3.51	& 7 & 3.26	\\
5 &  $8.2 \cdot 10^{6}$	& 58 & 33.71	& 29 & 25.41	& 16 & 22.98	& 11 & 22.02	& 9 & 23.02	& 7 & 21.57	\\
6 &  $6.6 \cdot 10^{7}$	& 56 & 213.39	& 28 & 165.31	& 15 & 152.04	& 11 & 157.86	& 8 & 148.09	& 7 & 158.87	\\
7 &  $5.3 \cdot 10^{8}$	& 54 & 1509.44	& 27 & 1189.60	& 15 & 1144.01 & 11 & 1194.04	& 8 & 1128.42	& 7 & 1210.43	\\
\hline \\[-1ex]
\hline
\multicolumn{2}{c|}{} & \multicolumn{6}{c|}{} & \multicolumn{6}{c}{} \\[-2.5ex]
\multicolumn{2}{c|}{} & \multicolumn{6}{c|}{$\mathcal{P}_\ell(\hat{A}_s, \widetilde{S}_s)$} & \multicolumn{6}{c}{$\mathcal{P}_\ell(\hat{A}_s, \widetilde{S})$}\\
\hline
\multicolumn{2}{c|}{$\nu$} & \multicolumn{2}{c}{4}  & \multicolumn{2}{c}{6} & \multicolumn{2}{c|}{8} & \multicolumn{2}{c}{4}  & \multicolumn{2}{c}{6} & \multicolumn{2}{c}{8}\\
$\ell$ & \multicolumn{1}{c|}{DoFs}	& \multicolumn{1}{c}{iter} & \multicolumn{1}{c}{time} & \multicolumn{1}{c}{iter} & \multicolumn{1}{c}{time} & \multicolumn{1}{c}{iter} & \multicolumn{1}{c|}{time} & \multicolumn{1}{c}{iter} & \multicolumn{1}{c}{time} & \multicolumn{1}{c}{iter} & \multicolumn{1}{c}{time} & \multicolumn{1}{c}{iter} & \multicolumn{1}{c}{time}\\
\hline 
& & & & & & &\multicolumn{1}{r|}{} & & & & \\[-2.5ex]
1 &  $1.4 \cdot 10^{3}$	& 10 & 0.08	 & 7 & 0.06	 & 6 & \multicolumn{1}{r|}{0.05}	 & 13 & 0.10	 & 9 & 0.07	 & 7 & 0.06\\	
2 &  $1.4 \cdot 10^{4}$	& 9  & 0.17	 & 7 & 0.15	 & 6 & \multicolumn{1}{r|}{0.14}	 & 12 & 0.21	 & 9 & 0.18	 & 7 & 0.15\\	
3 &  $1.2 \cdot 10^{5}$	& 10 & 0.55	 & 6 & 0.41	 & 5 & \multicolumn{1}{r|}{0.40}	 & 12 & 0.61	 & 8 & 0.51	 & 6 & 0.44\\	
4 &  $1.0 \cdot 10^{6}$	& 10 & 2.50	 & 7 & 2.31	 & 5 & \multicolumn{1}{r|}{2.07}	 & 11 & 2.44	 & 8 & 2.33	 & 6 & 2.16\\	
5 &  $8.2 \cdot 10^{6}$	& 11 & 17.00	 & 7 & 15.11	 & 5 & \multicolumn{1}{r|}{13.65}	 & 11 & 14.54	 & 8 & 14.58	 & 6 & 14.03\\	
6 &  $6.6 \cdot 10^{7}$	& 11 & 120.86	 & 7 & 109.35	 & 5 & \multicolumn{1}{r|}{101.26}	 & 11 & 102.66	 & 7 & 92.09	 & 6 & 101.90\\	
7 &  $5.3 \cdot 10^{8}$	& 11 & 913.00	 & 7 & 830.17	 & 5 & \multicolumn{1}{r|}{770.80}	 & 10 & 700.38	 & 7 & 693.75	 & 6 & 769.34\\
\hline
\end{tabular}
\caption{Iteration numbers and time-to-solution (in sec.) including the $W$-cycle for several refinement levels in case of $\mathcal{P}_\ell(\cdot, \cdot)$.}
\label{T:Wcycle_iter_Pl}
\end{table}
In \cref{T:Wcycle_iter_Pl}, we present the corresponding results. The computations are performed up to level $\ell=7$ which involves $5.3\cdot10^8$ DoFs. We observe level independent iteration numbers for all approaches. In case of $\mathcal{P}_\ell(\hat{A}_s, \hat{S})$, the time optimum seems to be located between $\nu=4$ and $\nu=8$. By applying a Gauss--Seidel scheme to $C$ the time-to-solution can be further reduced. In particular, we are able to reduce the time-to-solution by almost
40\% with $\widetilde{S}$ instead of $\hat{S}$.

\begin{table}[h!]
\centering
\setlength{\tabcolsep}{0.53mm}
\begin{tabular}{ l l | r r r r r r r r r r r r}
\hline
\multicolumn{2}{c|}{} & \multicolumn{12}{c}{} \\[-2.5ex]
\multicolumn{2}{c|}{} & \multicolumn{12}{c}{$\mathcal{P}_\ell(\hat{A}_s, \hat{S})$} \\
\hline
\multicolumn{2}{c|}{$\nu$} & \multicolumn{2}{c}{1}  & \multicolumn{2}{c}{2} & \multicolumn{2}{c}{4} & \multicolumn{2}{c}{6}  & \multicolumn{2}{c}{8} & \multicolumn{2}{c}{14}\\
$\ell$ & \multicolumn{1}{c|}{DoFs}	& \multicolumn{1}{c}{iter} & \multicolumn{1}{c}{time} & \multicolumn{1}{c}{iter} & \multicolumn{1}{c}{time} & \multicolumn{1}{c}{iter} & \multicolumn{1}{c}{time} & \multicolumn{1}{c}{iter} & \multicolumn{1}{c}{time} & \multicolumn{1}{c}{iter} & \multicolumn{1}{c}{time} & \multicolumn{1}{c}{iter} & \multicolumn{1}{c}{time}\\
\hline 
& & & & & & & & & & & \\[-2.5ex]
1 &  $1.4 \cdot 10^{3}$	& 66	& 0.24	& 50	& 0.20	& 17	& 0.08	& 12	& 0.06		& 9	& 0.04		& 7	& 0.05	\\
2 &  $1.4 \cdot 10^{4}$	& 63	& 0.32	& 31	& 0.19	& 18	& 0.14	& 12	& 0.11		& 10	& 0.11		& 6	& 0.10	\\
3 &  $1.2 \cdot 10^{5}$	& --	& --	& 30	& 0.52	& 18	& 0.46	& 12	& 0.41		& 9	& 0.39		& 6	& 0.42	\\
4 &  $1.0 \cdot 10^{6}$	&	&	&--	& --	& 18	& 2.75	& 12	& 2.58		& 9	& 2.49		& 7	& 3.24	\\
5 &  $8.2 \cdot 10^{6}$	& 	& 	& 	& 	& 22	& 24.13	& 11	& 17.11		& 9	& 18.21		& 7	& 23.76	\\
6 &  $6.6 \cdot 10^{7}$	& 	& 	& 	& 	& 118	& 990.37& 13	& 155.42	& 11	& 170.63	& 8	& 209.41	\\
7 &  $5.3 \cdot 10^{8}$	& 	& 	& 	& 	& --	& --	& 16	& 1468.48	& 13	& 1549.13	& 9	& 1811.18	\\
\hline \\[-1ex]
\hline
\multicolumn{2}{c|}{} & \multicolumn{6}{c|}{} & \multicolumn{6}{c}{} \\[-2.5ex]
\multicolumn{2}{c|}{} & \multicolumn{6}{c|}{$\mathcal{P}_\ell(\hat{A}_s, \widetilde{S}_s)$} & \multicolumn{6}{c}{$\mathcal{P}_\ell(\hat{A}_s, \widetilde{S})$}\\
\hline
\multicolumn{2}{c|}{$\nu$} & \multicolumn{2}{c}{4}  & \multicolumn{2}{c}{6} & \multicolumn{2}{c|}{8} & \multicolumn{2}{c}{4}  & \multicolumn{2}{c}{6} & \multicolumn{2}{c}{8}\\
$\ell$ & \multicolumn{1}{c|}{DoFs}	& \multicolumn{1}{c}{iter} & \multicolumn{1}{c}{time} & \multicolumn{1}{c}{iter} & \multicolumn{1}{c}{time} & \multicolumn{1}{c}{iter} & \multicolumn{1}{c|}{time} & \multicolumn{1}{c}{iter} & \multicolumn{1}{c}{time} & \multicolumn{1}{c}{iter} & \multicolumn{1}{c}{time} & \multicolumn{1}{c}{iter} & \multicolumn{1}{c}{time}\\
\hline 
& & & & & & &\multicolumn{1}{r|}{} & & & & \\[-2.5ex]
1 &  $1.4 \cdot 10^{3}$	& 10	& 0.04		& 8	& 0.04		& 7	& \multicolumn{1}{r|}{0.04}	& 13	& 0.05		& 10	& 0.05		& 8	& 0.04\\
2 &  $1.4 \cdot 10^{4}$	& 10	& 0.08		& 7	& 0.07		& 6	& \multicolumn{1}{r|}{0.07}	& 13	& 0.09		& 9	& 0.08		& 7	& 0.07\\
3 &  $1.2 \cdot 10^{5}$	& 9	& 0.25		& 7	& 0.25		& 6	& \multicolumn{1}{r|}{0.27}	& 12	& 0.28		& 8	& 0.25		& 8	& 0.32\\
4 &  $1.0 \cdot 10^{6}$	& 10	& 1.64		& 9	& 2.08		& 7	& \multicolumn{1}{r|}{2.09}	& 12	& 1.71		& 8	& 1.60		& 9	& 2.30\\
5 &  $8.2 \cdot 10^{6}$	& 11	& 13.08		& 11	& 18.65		& 7	& \multicolumn{1}{r|}{15.40}	& 14	& 14.24		& 10	& 14.31		& 10	& 18.52\\
6 &  $6.6 \cdot 10^{7}$	& 14	& 128.04	& 13	& 170.02	& 8	& \multicolumn{1}{r|}{136.09}	& 18	& 139.31	& 13	& 142.72	& 12	& 170.6\\
7 &  $5.3 \cdot 10^{8}$	& 19	& 1334.03	& 16	& 1606.78	& 8	& \multicolumn{1}{r|}{1045.98}	& 24	& 1417.72	& 17	& 1424.13	& 13	& 1409.50\\
\hline
\end{tabular}
\caption{Iteration numbers and time-to-solution (in sec.) including the $V$-cycle for several refinement levels in case of $\mathcal{P}_\ell(\cdot, \cdot)$.}
\label{T:Vcycle_iter_Pl}
\end{table}
Even though not covered by the theory in this article, 
we finally present numerical results including the $V$-cycle in \cref{T:Vcycle_iter_Pl}.
Note, the settings and parameters remain the same as for the previous results in \cref{T:Wcycle_iter_Pl}.
We observe that in case of a $V$-cycle the Uzawa smoother $\mathcal{P}_\ell(\hat{A}_s, \hat{S})$  diverges for small numbers of smoothing steps, as indicated by ''--'' in the table. Only with more smoothing steps convergence can be enforced. The convergence deteriorates and then more iterations are needed when the number of levels is increased. Choosing a different pressure smoother, here either $\widetilde{S}_s$ or $\widetilde{S}$, leads to better results in the iteration numbers, 
but may still lack of robustness. In the sequential setting used here, the time-to-solution achieved with the $W$-cycle and $V$-cycle are comparable, with the $W$-cycle having advantages for large problems. Note that this assessment is expected to change in a parallel setting, since a $W$-cycle with $\ell = 7$ levels will require $2^7=128$ coarse grid solves and may thus generate a significant sequential bottleneck. As a compromise, a variable $V$-cycle can be applied, where the number of smoothing steps is increased on the coarser levels of the mesh hierarchy. Experimentally this results in a robust method (wrt.\ iteration numbers) 
and an improved time-to-solution on massively parallel systems. For details, we refer to \cite{gmeiner-huber-john-ruede-wohlmuth_2015}.

\begin{remark}
The results of \cref{T:Wcycle_iter_Pl,T:Vcycle_iter_Pl} show that there is a fundamental difference if we use the mass matrix which is spectrally equivalent to the Schur complement or the scaled discrete pressure Laplacian for building the smoother in the pressure component. While formally we can satisfy condition \eqref{omega1} by using a hybrid Gauss-Seidel applied to either of these two options, the use of the scaled Laplacian is much more favorable. This might seem to be a surprise since $C$ is not spectrally equivalent to the Schur complement. However this choice mimics the fact that on the continuous level the commutator $\Delta \nabla - \nabla \Delta$ is zero, and thus the pressure satisfies a Laplace equation. A detailed study on the smoother performance taking into account the discrete commutator can be found in \cite{wang2013multigrid}. 
\end{remark}

\section{Concluding remarks}
In this article, we presented a unified analysis of the smoothing properties for different Uzawa variants. In particular, we were able to prove the smoothing property for the commonly used inexact Uzawa scheme $\mathcal{P}_\ell$. Also, a new smoother, belonging to the class of Uzawa schemes, has been proposed and analyzed. Numerical examples illustrated the obtained theoretical results and the computational cost of the different smoothers was discussed.

So far, the analysis requires a symmetric smoother $\hat{S}$ for the (inexact) Schur complement and additionally requires symmetry for $\hat{A}$ in the cases of $\mathcal{P}_i$, $i \in \{ \ell, u, f \}$. The extension to the non-symmetric case for $\hat{A}$ and $\hat{S}$, might be of great interest since the computational cost could perhaps be further narrowed.

\section*{Acknowledgments}
This work was financially supported (in part) by the German Research Foundation (DFG)  through  grant WO 671/11-1.

\bibliographystyle{plain}
\bibliography{references}

\end{document}